\documentclass[10pt,a4paper,reqno,sumlimits]{amsart}
\usepackage[margin=3.5cm]{geometry}
\usepackage[numbers]{natbib}
\usepackage{amsmath}
\usepackage{amssymb}
\usepackage[mathscr]{euscript}
\usepackage{enumerate}
\usepackage{xspace}
\usepackage{color}
\usepackage{enumerate}
\usepackage{enumitem}
\usepackage{amsthm}

\begin{document}

\theoremstyle{plain}
\newtheorem{C}{Convention}
\newtheorem{SA}{Standing Assumption}[section]
\newtheorem{SAS}{Standing Assumption}[subsection]
\newtheorem{theorem}{Theorem}[section]
\newtheorem{condition}{Condition}[section]
\newtheorem{lemma}[theorem]{Lemma}
\newtheorem{fact}[theorem]{Fact}
\newtheorem{proposition}[theorem]{Proposition}
\newtheorem{corollary}[theorem]{Corollary}
\newtheorem{claim}[theorem]{Claim}
\newtheorem{definition}[theorem]{Definition}
\newtheorem{Ass}[theorem]{Assumption}
\newcommand{\q}{Q}
\theoremstyle{definition}
\newtheorem{remark}[theorem]{Remark}
\newtheorem{note}[theorem]{Note}
\newtheorem{example}[theorem]{Example}
\newtheorem{assumption}[theorem]{Assumption}
\newtheorem*{notation}{Notation}
\newtheorem*{assuL}{Assumption ($\mathbb{L}$)}
\newtheorem*{assuAC}{Assumption ($\mathbb{AC}$)}
\newtheorem*{assuEM}{Assumption ($\mathbb{EM}$)}
\newtheorem*{assuES}{Assumption ($\mathbb{ES}$)}
\newtheorem*{assuM}{Assumption ($\mathbb{M}$)}
\newtheorem*{assuMM}{Assumption ($\mathbb{M}'$)}
\newtheorem*{assuL1}{Assumption ($\mathbb{L}1$)}
\newtheorem*{assuL2}{Assumption ($\mathbb{L}2$)}
\newtheorem*{assuL3}{Assumption ($\mathbb{L}3$)}
\newtheorem{charact}[theorem]{Characterization}
\newcommand{\notiz}{\textup} 
\renewenvironment{proof}{{\parindent 0pt \it{ Proof:}}}{\mbox{}\hfill\mbox{$\Box\hspace{-0.5mm}$}\vskip 16pt}
\newenvironment{proofthm}[1]{{\parindent 0pt \it Proof of Theorem #1:}}{\mbox{}\hfill\mbox{$\Box\hspace{-0.5mm}$}\vskip 16pt}
\newenvironment{prooflemma}[1]{{\parindent 0pt \it Proof of Lemma #1:}}{\mbox{}\hfill\mbox{$\Box\hspace{-0.5mm}$}\vskip 16pt}
\newenvironment{proofcor}[1]{{\parindent 0pt \it Proof of Corollary #1:}}{\mbox{}\hfill\mbox{$\Box\hspace{-0.5mm}$}\vskip 16pt}
\newenvironment{proofprop}[1]{{\parindent 0pt \it Proof of Proposition #1:}}{\mbox{}\hfill\mbox{$\Box\hspace{-0.5mm}$}\vskip 16pt}
\newcommand{\s}{\u s}

\newcommand{\Law}{\ensuremath{\mathop{\mathrm{Law}}}}
\newcommand{\loc}{{\mathrm{loc}}}
\newcommand{\Log}{\ensuremath{\mathop{\mathscr{L}\mathrm{og}}}}
\newcommand{\Meixner}{\ensuremath{\mathop{\mathrm{Meixner}}}}
\newcommand{\of}{[\hspace{-0.06cm}[}
\newcommand{\gs}{]\hspace{-0.06cm}]}

\let\MID\mid
\renewcommand{\mid}{|}

\let\SETMINUS\setminus
\renewcommand{\setminus}{\backslash}

\def\stackrelboth#1#2#3{\mathrel{\mathop{#2}\limits^{#1}_{#3}}}

\renewcommand{\theequation}{\thesection.\arabic{equation}}
\numberwithin{equation}{section}

\newcommand\llambda{{\mathchoice
      {\lambda\mkern-4.5mu{\raisebox{.4ex}{\scriptsize$\backslash$}}}
      {\lambda\mkern-4.83mu{\raisebox{.4ex}{\scriptsize$\backslash$}}}
      {\lambda\mkern-4.5mu{\raisebox{.2ex}{\footnotesize$\scriptscriptstyle\backslash$}}}
      {\lambda\mkern-5.0mu{\raisebox{.2ex}{\tiny$\scriptscriptstyle\backslash$}}}}}

\newcommand{\prozess}[1][L]{{\ensuremath{#1=(#1_t)_{0\le t\le T}}}\xspace}
\newcommand{\prazess}[1][L]{{\ensuremath{#1=(#1_t)_{0\le t\le T^*}}}\xspace}

\newcommand{\tr}{\operatorname{tr}}
\newcommand{\lijepoa}{{\mathscr{A}}}
\newcommand{\lijepob}{{\mathscr{B}}}
\newcommand{\lijepoc}{{\mathscr{C}}}
\newcommand{\lijepod}{{\mathscr{D}}}
\newcommand{\lijepoe}{{\mathscr{E}}}
\newcommand{\lijepof}{{\mathscr{F}}}
\newcommand{\lijepog}{{\mathscr{G}}}
\newcommand{\lijepok}{{\mathscr{K}}}
\newcommand{\lijepoo}{{\mathscr{O}}}
\newcommand{\lijepop}{{\mathscr{P}}}
\newcommand{\lijepoh}{{\mathscr{H}}}
\newcommand{\lijepom}{{\mathscr{M}}}
\newcommand{\lijepou}{{\mathscr{U}}}
\newcommand{\lijepov}{{\mathscr{V}}}
\newcommand{\lijepoy}{{\mathscr{Y}}}
\newcommand{\cF}{{\mathscr{F}}}
\newcommand{\cG}{{\mathscr{G}}}
\newcommand{\cH}{{\mathscr{H}}}
\newcommand{\cM}{{\mathscr{M}}}
\newcommand{\cD}{{\mathscr{D}}}
\newcommand{\bD}{{\mathbb{D}}}
\newcommand{\bF}{{\mathbb{F}}}
\newcommand{\bG}{{\mathbb{G}}}
\newcommand{\bH}{{\mathbb{H}}}
\newcommand{\dd}{\operatorname{d}\hspace{-0.05cm}}
\newcommand{\ddd}{\operatorname{d}}
\newcommand{\er}{{\mathbb{R}}}
\newcommand{\ce}{{\mathbb{C}}}
\newcommand{\erd}{{\mathbb{R}^{d}}}
\newcommand{\en}{{\mathbb{N}}}
\newcommand{\de}{{\mathrm{d}}}
\newcommand{\im}{{\mathrm{i}}}
\newcommand{\indik}{{\mathbf{1}}}
\newcommand{\D}{{\mathbb{D}}}
\newcommand{\E}{E}
\newcommand{\N}{{\mathbb{N}}}
\newcommand{\Q}{{\mathbb{Q}}}
\renewcommand{\P}{{\mathbb{P}}}
\newcommand{\ud}{\operatorname{d}\!}
\newcommand{\ii}{\operatorname{i}\kern -0.8pt}
\newcommand{\cadlag}{c\`adl\`ag }
\newcommand{\p}{P}
\newcommand{\F}{\mathbf{F}}
\newcommand{\1}{\mathbf{1}}
\newcommand{\f}{\mathscr{F}^{\hspace{0.03cm}0}}
\newcommand{\lle}{\langle\hspace{-0.085cm}\langle}
\newcommand{\rre}{\rangle\hspace{-0.085cm}\rangle}
\newcommand{\llbr}{[\hspace{-0.085cm}[}
\newcommand{\rrbr}{]\hspace{-0.085cm}]}
\newcommand{\g}{|\hspace{-0.06cm}|}

\def\EM{\ensuremath{(\mathbb{EM})}\xspace}

\newcommand{\la}{\langle}
\newcommand{\ra}{\rangle}

\newcommand{\Norml}[1]{%
{|}\kern-.25ex{|}\kern-.25ex{|}#1{|}\kern-.25ex{|}\kern-.25ex{|}}

\title[Absolute Continuity of Semimartingales]{Absolute Continuity of Semimartingales} 
\author[D. Criens]{David Criens}
\address{D. Criens - Technical University of Munich, Center for Mathematics, Germany}
\email{david.criens@tum.de}
\author[K. Glau]{Kathrin Glau}
\address{K. Glau - Technical University of Munich, Center for Mathematics, Germany}
\email{kathrin.glau@tum.de}
\keywords{absolute continuity of laws, semimartingale, Girsanov's theorem, change of measure, martingale problem, explosion\vspace{1ex}}



\date{\today}
\maketitle

\frenchspacing
\pagestyle{myheadings}

\begin{abstract}
We derive equivalent conditions for the (local) absolute continuity of two laws of semimartingales on random sets.	
Our result generalizes previous results for classical semimartingales by replacing a strong uniqueness assumption by a weaker uniqueness assumption.
The main tool is a generalized Girsanov's theorem, which relates laws of two possibly explosive semimartingales to a candidate density process. 
Its proof is based on an extension theorem for consistent families of probability measures.
Moreover, we show that in a one-dimensional It\^o-diffusion setting our result reproduces the known deterministic characterizations for (local) absolute continuity. 
Finally, we give a Khasminskii-type test for the absolute continuity of multi-dimensional It\^o-diffusions and derive linear growth conditions for the martingale property of stochastic exponentials.
\end{abstract}

\section{Introduction}
In the 1970s, probabilists studied conditions under which laws of semimartingales are (locally) absolutely continuous. The most general results were obtained by Jacod and M\'emin \cite{JM76} and Kabanov, Lipster and Shiryaev \cite{KLS-LACOM1,KLS-LACOM2} under a strong uniqueness assumption, called local uniqueness in the monograph of Jacod and Shiryaev \cite{JS}.

In this article we provide equivalent statements for the (local) absolute continuity of semimartingales on random sets under the assumption that the dominated law is unique.
While in Markovian settings local uniqueness is implied by uniqueness, it is surprising that this weaker condition suffices also beyond Markovian setups.

Our main tool is a generalized version of Girsanov's theorem for semimartingales, which relates two laws of semimartingales on random sets through a local martingale density. 
Key of the proof is to replace the classical Skorokhod space by a slightly larger path space whose topological properties allow the extension of relevant consistent families of probability measures. 

Let us highlight related result from the literature.
Under the so-called Engelbert-Schmidt conditions, a deterministic characterization of the (local) absolute continuity of one-dimensional It\^o-diffusions was given by Cherny and Urusov \cite{Cherny2006}.
In a similar setting, Mijativi\'c and Urusov \cite{MU(2012)} proved equivalent conditions for the martingale property of stochastic exponentials. In both cases, the proofs are based on an extension of stopping times and different from ours.
We relate our main result to these observations and explain that the deterministic characterizations also follow from our main result, see Section \ref{sec: Diff} below. In other words, we provide alternative proofs for the results. 
In an It\^o-jump-diffusion setting, Cheridito, Filipovi\'c and Yor \cite{CFY} proved local absolute continuity if the dominated measure is unique and non-explosive. In Section \ref{sec: CFY} below, we explain the relation of their result to ours.
In a multidimensional It\^o-diffusion setting, Ruf \cite{RufSDE} proved equivalent conditions for the martingale property of stochastic exponentials using an extension argument similar to ours. 
The result can be deduced from ours, see Section \ref{sec: Diff M} below.

We also present two novel applications of our main result. First, we give deterministic conditions for (local) absolute continuity of multidimensional It\^o-diffusions, extending the work of Ruf \cite{RufSDE}.
The idea is similar to Khasminskii's test for explosion, i.e. using comparison arguments we reduce the question when two multidimensional It\^o-diffusions are (locally) absolutely continuous to the question when an integral functional of a one-dimensional It\^o-diffusion converges, see Section \ref{sec: Diff M} below.
As a second application of our main result, 
we generalize Bene\u s's \cite{doi:10.1137/0309034} linear growth condition for the martingale property of stochastic exponentials to continuous It\^o-process drivers.
We emphasis that this application differs from the others, because no uniqueness argument is necessary. 

Let us also comment on further related literature. 
An extension argument similar to ours was used by Ruf and Perkovski \cite{perkowski2015} to study F\"ollmer measures, and by Kardaras, Kreher and Nikeghbali \cite{kardaras2015} to study the influence of strict local martingales on pricing financial derivatives.

The article is structured as follows. In Section \ref{sec: main} we introduce our setting and present our main results. Criteria for absolute continuity of semimartingales are studied in Section \ref{sec:LE} and in Section \ref{sec: comment} we relate our results to those in \cite{CFY, MU(2012)}. Finally, in Section \ref{sec: Diff M} we discuss conditions for the absolute continuity of multidimensional diffusions and in Section \ref{sec: BC} we derive criteria for the martingale property of stochastic exponentials.

Let us end the introduction with a remark on notation:
All non-explained notation can be found in the monograph of Jacod and Shiryaev \cite{JS}. Furthermore, all standing assumptions are imposed only for the section they are stated in. 

\section{A Generalized Girsanov Theorem}\label{sec: main}
\label{Hilbert-Vlaued Semimartingales and Semimartingale Problems on Stochastic Intervals}
We start by introducing our probabilistic setup.
We adjoint an isolated point \(\Delta\) to \(\mathbb{R}^d\) and write \(\mathbb{R}^d_\Delta \triangleq \mathbb{R}^d \cup\{\Delta\}\). For a function \(\alpha \colon \mathbb{R}_+ \to \mathbb{R}_\Delta^d\) we define \(\tau_{\Delta} (\alpha) \triangleq \inf(t \geq 0 \colon \alpha (t) = \Delta)\).
Let \(\Omega\) to be the set of all functions \(\alpha\colon [0, \infty) \to \mathbb{R}^d_\Delta\) such that \(\alpha\) is \cadlag on \([0, \tau_\Delta(\alpha))\) and \(\alpha (t) = \Delta\) for all \(t \geq \tau_{\Delta}(\alpha)\). 
Let \(X_t(\alpha) = \alpha(t)\) be the coordinate process and define \(\mathscr{F} \triangleq \sigma(X_t, t \geq 0)\). Moreover, for each \(t \geq 0\) we define \(\mathscr{F}^o_t\triangleq \sigma(X_s, s \in [0, t])\) and \(\mathscr{F}_t \triangleq \bigcap_{s > t} \mathscr{F}^o_s\). We work with the right-continuous filtration \(\F \triangleq (\mathscr{F}_t)_{t \geq 0}\). 

In general, if we use terms such as local martingale, semimartingale, stopping time, predictable, etc. we refer to \(\F\) as the underlying filtration.

Note that for all \(t \geq 0\)
\[
\{\tau_{\Delta} \leq t\} = \{X_t = \Delta\} \in \mathscr{F}^o_t \subseteq \mathscr{F}_t,
\]
which implies that \(\tau_\Delta\) is a stopping time.

For a stopping time \(\xi\) we set 
\begin{align*}
\mathscr{F}_\xi \triangleq \{A \in \mathscr{F} \colon A \cap \{\xi \leq t\} \in \mathscr{F}_t\},
\end{align*}
and 
\[\mathscr{F}_{\xi -} \triangleq \sigma \left(\mathscr{F}^o_0, \{A \cap \{\xi > t\} \colon t \geq 0, A \in \mathscr{F}_t\}\right).\]
We note that in the second definition the treatment of the initial \(\sigma\)-field is different from the classical definition, where \(\mathscr{F}_0\) is used instead of \(\mathscr{F}^o_0\). In our case, \(\mathscr{F}_{\xi-}\) is countably generated, see \cite[Lemma E.1]{perkowski2015}, which is important for the extension argument in the proof of our fist main result, Theorem \ref{theo:main1} below.

The following facts for \(\mathscr{F}_{\xi-}\) can be verified as in the classical case:
 \begin{enumerate}
 	\item[\textup{(a)}]\(\mathscr{F}_{\xi-} \subseteq \mathscr{F}_\xi\).
 	\item[\textup{(b)}] For two stopping times \(\xi\) and \(\rho\) and any \(G \in \mathscr{F}_{\xi}\)
we have 
\(G \cap \{\rho > \xi\} \in \mathscr{F}_{\rho-}\) and for all \(G \in \mathscr{F}\) we have \(G \cap \{\rho=\infty\} \in \mathscr{F}_{\rho-}\).
\item[\textup{(c)}] For an increasing sequence \((\rho_n)_{n \in \mathbb{N}}\) of stopping times with \(\rho \triangleq \lim_{n \to \infty} \rho_n\) it holds that
\begin{align*}
\bigvee_{n \in \mathbb{N}} \mathscr{F}_{\rho_n -} = \mathscr{F}_{\rho-}.
\end{align*}
\end{enumerate}

For two stopping times \(\xi\) and \(\rho\) we define the stochastic interval
\[
\of \xi, \rho\gs \triangleq \{(\omega, t) \in \Omega \times [0, \infty) \colon \xi(\omega) \leq t \leq \rho(\omega)\}.
\]
All other stochastic intervals are defined in the same manner.


In the spirit of stochastic differential equations up to explosion, we now formulate a semimartingale problem up to explosion.
We start by introducing the parameters:
\begin{enumerate}
\item[\textup{(i)}]
Let \((B, C, \nu)\) be a so-called \textit{candidate triplet} 
consisting of
\begin{enumerate}
\item[--] a predictable \(\mathbb{R}^d_\Delta\)-valued process \(B\), 
\item[--] a predictable \(\mathbb{R}^d_\Delta \otimes \mathbb{R}^d_\Delta\)-valued process \(C\), 
\item[--] a predictable random measure \(\nu\) on \([0, \infty) \times \mathbb{R}^d\).
\end{enumerate}
\item[\textup{(ii)}]
Let \(\eta\) be a probability measure on \((\mathbb{R}^d, \mathscr{B}(\mathbb{R}^d))\), which we call \textit{initial law}.
\item[\textup{(iii)}] Let \(\rho\) be a stopping time, which we call \emph{lifetime}.
\end{enumerate}
We fix a truncation function \(h\) and suppose that all terms such as \emph{semimartingale characteristics} refer to this truncation function.

The idea of the semimartingale problem formulated below is to find a probability measure on \((\Omega, \mathscr{F})\) such that the coordinate process \(X\) is a semimartingale with characteristics \((B, C, \nu)\) up to the lifetime \(\rho\) and with initial law \(\eta\).

\begin{definition}\label{Definition Semimartingale Problem}
We call a probability measure \(\p\) on \((\Omega, \mathscr{F})\) a \emph{solution to the semimartingale problem (SMP)} associated with \((\rho; \eta; B, C, \nu)\), if there exists an increasing sequence \((\rho_n)_{n \in \mathbb{N}}\) of stopping times and a sequence of \(P\)-semimartingales \((X^n)_{n \in \mathbb{N}}\) such that \(\rho_n \nearrow \rho\) as \(n \to \infty\) and for all \(n \in \mathbb{N}\) the following holds:
\begin{enumerate}
\item[\textup{(i)}] the stopped process \(X^{\rho_n} \triangleq (X_{t \wedge \rho_n})_{t \geq 0}\) is \(\p\)-indistinguishable from \(X^n\),
\item[\textup{(ii)}] the \(P\)-characteristics of \(X^n\)
are \(\p\)-indistinguishable from the stopped triplet \((B^{\rho_n}, C^{\rho_n}, \nu^{\rho_n})\), where 
\[
 \nu^{\rho_n}(\omega, \dd t \times \dd x) 
\triangleq \1_{\of 0, \rho_n\gs \times \mathbb{R}^d} (\omega, t, x) \nu(\omega, \dd t \times \dd x),
\]
\item[\textup{(iii)}] \(\p \circ X^{-1}_0 = \eta\).
\end{enumerate}
The sequence \((\rho_n)_{n \in \mathbb{N}}\) is called \emph{\(\rho\)-localization sequence} and the sequence \((X^n)_{n \in \mathbb{N}}\) is called \emph{fundamental sequence}.
If \(\p(\rho = \infty) = 1\), we say that \(\p\) is \emph{conservative}.
\end{definition}
In a conservative setting the semimartingale problem was first introduced by Jacod \cite{J79}.

In this section we impose the following standing assumption.
\begin{SA}\label{SA1}
The probability measure \(\p\) is a solution to the SMP \((\rho; \eta; B, C, \nu)\) with \(\rho\)-localization sequence \((\rho_n)_{n \in \mathbb{N}}\) and fundamental sequence \((X^n)_{n \in \mathbb{N}}\), and \(Z\) is a non-negative local \(P\)-martingale such that \(E^P[Z_0] = 1\) and \((\sigma_n)_{n \in \mathbb{N}}\) is an increasing sequence of stopping times such that \(Z^{\sigma_n}\) is a uniformly integrable \(P\)-martingale. Furthermore, \(P\)-a.s. \(\sigma_n < \sigma \triangleq \lim_{n \to \infty} \sigma_n\). W.l.o.g. \(\rho_n \vee \sigma_n \leq n\).
\end{SA}
Of course, since we assume that \(\sigma_n \leq n\), the stopped process \(Z^{\sigma_n}\) is a uniformly integrable \(P\)-martingale whenever it is a \(P\)-martingale. 

Let us further comment on this standing assumption. Our aim is to relate \(P\) and \(Z\) to another solution of an SMP. We start with a local relation and define a sequence \((Q_n)_{n \in \mathbb{N}}\) of probability measures via \begin{align}\label{eq: Qn expl} Q_n \triangleq Z_{\sigma_n} \cdot P,\end{align}
which means \(Q_n(G) = E^P\left[Z_{\sigma_n} \1_G\right]\) for all \(G \in \mathscr{F}\). 
Each \(Q_n\) solves an SMP by Girsanov's theorem. The next step is to extend this sequence and show that the extension also solves an SMP.
We observe that the sequence \((Q_n)_{n \in \mathbb{N}}\) is consistent and consequently classical extension arguments yield that we find a probability measure \(Q\) such that \(Q = Q_n\) on \(\mathscr{F}_{\sigma_n-}\) for all \(n \in \mathbb{N}\). Next, we want to conclude that \(Q\) solves an SMP.
 For this aim, however, the identity \(Q = Q_n\) on \(\mathscr{F}_{\sigma_n-}\) is not sufficient. At this point, the last part of our standing assumption comes into play.
We sketch the idea. The details are given in the proof of Theorem \ref{theo:main1} below.
 For any \(G \in \mathscr{F}_{\sigma_n}\) we have \(G \cap \{\sigma_n < \sigma_m\}\in \mathscr{F}_{\sigma_n}\cap\mathscr{F}_{\sigma_m-}\) and
\begin{align*}
Q (G \cap \{\sigma_n < \sigma_m\}) 
&= E^P \left[ Z_{\sigma_n} \1_{G \cap \{\sigma_n < \sigma_m\}}\right].
\end{align*}
Letting \(m \nearrow \infty\) and using our assumption that \(P\)-a.s. \(\sigma_n < \sigma\), we see that \(Q_n = Q\) on \(\mathscr{F}_{\sigma_n}\). This identity allows us to conclude that \(Q\) solves a SMP.

In the following two remarks we comment on choices for \((\sigma_n)_{n \in \mathbb{N}}\) and explain how to construct \(Z\) from a non-negative local \(P\)-martingale, which is only defined on a random set.
\begin{remark}\label{rem: sigma}
	An example for the sequence \((\sigma_n)_{n \in \mathbb{N}}\) in Standing Assumption \ref{SA1} is 
	\begin{align*}
	\sigma_n \triangleq \inf(t \geq 0 \colon Z_t > n) \wedge n. 
	\end{align*} 
	To see this, it suffices to note that \(Z_{t \wedge \sigma_n} \leq n + Z_{\sigma_n}\). Since \(Z_{\sigma_n}\) is \(P\)-integrable by Fatou's lemma, \(Z^{\sigma_n}\) is a uniformly integrable \(P\)-martingale by the dominated convergence theorem. Furthermore, in this case \(\{\sigma= \infty\}\) and \(\{\sigma_n < \sigma\}\) are \(P\)-full sets. More generally, \(\sigma_n\) can be chosen as \(\gamma_n \wedge n\), where \((\gamma_n)_{n \in \mathbb{N}}\) is a \(P\)-localizing sequence for \(Z\).
\end{remark}

\begin{remark}\label{rem: extension}
	Let \((\xi_n)_{n \in \mathbb{N}}\) be an increasing sequence of stopping times. We say that a process \(\widehat{Z}\) is a non-negative local \(P\)-martingale on the random set \(\bigcup_{n \in \mathbb{N}} \of 0, \xi_n\gs\), if
	the stopped process \(\widehat{Z}^{\xi_n}\) is a non-negative local \(P\)-martingale.
	It is always possible to extend the process to a globally defined non-negative local \(P\)-martingale by setting 
	\begin{align}\label{eq: extension Z}
	Z \triangleq \begin{cases}
	\widehat{Z},&\textup{ on } \bigcup_{n \in \mathbb{N}} \of 0, \xi_n\gs,\\
	\liminf_{n \to \infty} \widehat{Z}_{\xi_n},&\textup{ otherwise}.
	\end{cases}
	\end{align}
	By Fatou's lemma, the extension \(Z\) is a \(P\)-supermartingale and consequently the terminal value is \(P\)-a.s. finite.
	Using the Doob-Meyer decomposition theorem for supermartingales, it can be shown that \(Z\) is even a local \(P\)-martingale, see \cite[Lemma 12.43]{J79}.
\end{remark}

So far we have explained that we want to relate \(P\) and \(Z\) to a solution of an SMP. Our next step is to formally introduce the parameters of the SMP to which we want to connect \(P\) and \(Z\).

For \(n \in \mathbb{N}\) denote by \(X^{c, n}\) the continuous local \(P\)-martingale part of \(X^n\) and by \(Z^{c}\) the continuous local \(P\)-martingale part of \(Z\). Both are unique up to \(P\)-indistinguishability. 
The predictable quadratic covariation process (w.r.t. \(P\)) is denoted by \(\lle \cdot, \cdot\rre\).
Finally, let \((B', C, \nu')\) be a candidate triplet, such that up to a \(P\)-null set
\begin{equation}\label{candidtate triplet main theorem}
\begin{split}
B' &= B + \sum_{k = 0}^\infty \int_0^{\cdot}\frac{\1\{Z_{s-} > 0\}}{Z_{s-}} \1\{\rho_k \leq s < \rho_{k +1}\} \dd\hspace{0.06cm} \lle Z^c, X^{c, k} \rre_s + h(x) (Y - 1) \star \nu,\\ \nu' &= Y \cdot \nu
\end{split}
\end{equation} 
with \(\rho_0 \triangleq 0\),
\begin{align*}
Y \triangleq \frac{\1\{Z_- > 0\}}{Z_-}M^\p_{\mu^X} \left(Z \big|\widetilde{\mathscr{P}}\right),
\end{align*}
and \(M^P_{\mu^X}(\cdot|\widetilde{\mathscr{P}})\) denoting the conditional expectation w.r.t. the Dol\'eans measure
\begin{align*}
M^P_{\mu^X}(\dd \omega \times \dd t\times \dd x) \triangleq \mu^X(\omega, \dd t \times \dd x) P(\dd \omega)
\end{align*}
conditioned on \(\widetilde{\mathscr{P}} \triangleq \mathscr{P} \otimes \mathscr{B}(\mathbb{R}^d)\), see \cite[Section III.3.c)]{JS} for more details.
Here, we use the notation
\begin{align*}
h(x)(Y - 1) \star \nu_{t} \triangleq \begin{cases}
\int_0^{t}\int h(x) (Y(s, x) - 1) \nu(\dd s \times \dd x),&\textup{if it converges},\\
\Delta,&\textup{otherwise,}
\end{cases}
\end{align*}
and 
\begin{align*}
(Y \cdot \nu) (\dd t \times \dd x) \triangleq Y(t, x) \nu(\dd t \times \dd x).
\end{align*}
Furthermore, we use the convention that \(\Delta + x \equiv \Delta\) for all \(x \in \mathbb{R}^d_\Delta\).

Let us shortly comment on the intuition behind the modified candidate triplet \((B', C, \nu')\). The idea is to consider the probability measure \(Q_n\) as defined in \eqref{eq: Qn expl}. Then, by Girsanov's theorem, the stopped process \(Y^{n}_{\cdot \wedge \rho_n \wedge \sigma_n}\) is a  \(Q_n\)-semimartingale whose characteristics are \(Q_n\)-indistinguishable from the stopped modified triplet \((B'_{\cdot \wedge \rho_n \wedge \sigma_n}, C_{\cdot \wedge \rho_n \wedge \sigma_n}, \1_{\of 0, \rho_n \wedge \sigma_n\gs} \cdot \nu')\). Thus, if an extension of \(Q_n\) solves a SMP, the corresponding candidate triplet should be \((B', C, \nu')\).

For a second probability measure \(Q\) on \((\Omega, \mathscr{F})\), we write \(Q \ll_\textup{loc} P\) if \(Q \ll P\) on \(\mathscr{F}_{t}\) for all \(t \geq 0\). Moreover, we set 
\begin{align*}
\zeta \triangleq \sigma \wedge \rho.
\end{align*}
We are now in the position to state our first main result.
\begin{theorem}\label{theo:main1}
There exists a solution \(Q\) to the SMP \((\zeta; \eta'; B', C, \nu')\), where \[\eta' (G) \triangleq E^P\left[Z_0 \1_{\{X_0 \in G\}}\right]\] for \(G \in \mathscr{B}(\mathbb{R}^d)\), and 
\begin{align}\label{eq: eqality sigma}
Q = Z_{\sigma_n} \cdot P \text{ on } \mathscr{F}_{\sigma_n} \text{ for all } n \in \mathbb{N}.
\end{align}
Moreover, the following holds:
\begin{enumerate}
\item[\textup{(a)}] 
For all stopping times \(\xi\) we have
\begin{align}\label{eq:genGir}
\q = Z_\xi \cdot P \text{ on } \mathscr{F}_{\xi} \cap \{\sigma > \xi\}.\end{align}
\item[\textup{(b)}] The following are equivalent:
\begin{enumerate}
	\item[\textup{(b.i)}]
	\(Q\)-a.s. \(\sigma = \infty\).
	\item[\textup{(b.ii)}]
	The process \(Z\) is a \(P\)-martingale and \(P\)-a.s. \(Z = 0\) on \(\of \sigma, \infty\of\).
\end{enumerate}
If these statements hold true, then \(Q \ll_{\textup{loc}} P\) with \(\frac{\dd Q}{\dd P}|_{\mathscr{F}_{t}} = Z_t\) for all \(t \geq 0\).
\item[\textup{(c)}]
The following are equivalent:
\begin{enumerate}
	\item[\textup{(c.i)}]
	There exists an increasing sequence  \((\gamma_n)_{n \in \mathbb{N}}\) of stopping times such that \(\gamma_n \nearrow \sigma\) as \(n \to \infty\), \(Z^{\gamma_n}\) is a uniformly integrable \(P\)-martingale and 
	\begin{align*}
	\lim_{n \to \infty} Q\left(\gamma_n = \sigma = \infty \right)= 1.
	\end{align*}
	\item[\textup{(c.ii)}]
	The process \(Z\) is a uniformly integrable \(P\)-martingale with \(P\)-a.s. \(Z = 0\) on \(\of \sigma, \infty\of\).
\end{enumerate}
If these statements hold true, then \(Q \ll P\) with \(\frac{\dd Q}{\dd P} = \lim_{t \to \infty} Z_t \triangleq Z_\infty\).
\item[\textup{(d)}]
Suppose that at least one of the following conditions holds:
\begin{enumerate}
		\item[\textup{(d.i)}] \(Q\)-a.s. \(\rho_n < \sigma\) for all \(n \in \mathbb{N}\).
	\item[\textup{(d.ii)}]  \(P\)-a.s. \(\rho_n < \sigma\) and \(E^P\big[Z_{\rho_n}\big] = 1\) for all \(n \in \mathbb{N}\).
	\end{enumerate}
	Then, \(Q\) solves the SMP \((\rho; \eta'; B', C, \nu')\).
\end{enumerate}
\end{theorem}
We stress that the terminal random variable \(Z_\infty\) is \(P\)-a.s. well-defined due to the supermartingale convergence theorem.

Part (a) of this theorem is a Girsanov-type formula, part (b) gives a criterion for the local absolute continuity of \(Q\) and \(P\) and part (c) gives a criterion for the global absolute continuity. 
In part (d) we give conditions such that \(Q\) solves an SMP with lifetime \(\rho\). In this case, our observations from (b) and (c) give criteria for the (local) absolute continuity of solutions of two SMPs with the same lifetime. In (d) we present a condition which only depends on \(Q\) and a condition which only depends on \(P\). The latter is important for applications because it allows us to check properties of \(P\) to conclude that \(Q\) solves an SMP with lifetime \(\rho\). The condition \(E^P \big[Z_{\rho_n}\big] = 1\) means that the stopped process \(Z^{\rho_n}\) is a uniformly integrable \(P\)-martingale.

\begin{remark}\label{rem: vanish}
	If \(P\)-a.s. \(Z = 0\) on \(\of \sigma, \infty\of\), then (b.i) and (b.ii) in Theorem \ref{theo:main1} are equivalent to \(Q \ll_{\textup{loc}} P\) with \(\frac{\dd Q}{\dd P}|_{\mathscr{F}_{t}} = Z_t\) for all \(t \geq 0\), and (c.i) and (c.ii) in Theorem \ref{theo:main1} are equivalent to \(Q \ll P\) with \(\frac{\dd Q}{\dd P} = Z_\infty\).
\end{remark}
We would like to choose \((\sigma_n)_{n \in \mathbb{N}}\) such that \(P\)-a.s. \(Z = 0\) on \(\of \sigma, \infty\of\). 
Of course, this is the case if \(P\)-a.s. \(\sigma = \infty\), which is true when \(\sigma_n\) is chosen as proposed in Remark \ref{rem: sigma}. In particular, it is interesting to note that when \(P\)-a.s. \(\sigma = \infty\), then \((\sigma_n)_{n \in \mathbb{N}}\) is a \(P\)-localization sequence for the local \(P\)-martingale \(Z\).

Let us mention another natural choice for \((\sigma_n)_{n \in \mathbb{N}}\). 
Suppose that
\[
Z \triangleq \exp \left(U - \tfrac{1}{2} \lle U, U\rre\right),
\]
where \(U\) is a continuous local \(P\)-martingale. Set
\[
\sigma_n \triangleq \inf(t \geq 0 \colon \lle U, U\rre_t \geq n) \wedge n,
\]
then \(P\)-a.s. \(Z = 0\) on \(\of \sigma, \infty\of\), which follows from the strong law of large numbers for continuous local martingales, see \cite[Exercise V.1.16]{RY}. 
In view of Theorem \ref{theo:main1}, this choice of \((\sigma_n)_{n \in \mathbb{N}}\) shows that \(Q \ll_\textup{loc} P\) is equivalent to \(Q\)-a.s. \(\lle U, U\rre_t < \infty\) for all \(t \geq 0\), and that \(Q \ll P\) is equivalent to \(Q\)-a.s. \(\lle U, U\rre_\infty < \infty\). This observation is in the spirit of classical results for the local absolute continuity of globally defined semimartingales. We comment on this in Section \ref{sec:LE} below, where we also define a version of \(\sigma_n\) in the presence of jumps.
\\\\
\noindent
\textit{Proof of Theorem \ref{theo:main1}:}
We construct \(\q\) using the extension theorem of Parthasarathy. 
We recall a definition due to F\"ollmer \cite{follmer72}.
Let \(\mathbb{T} \subseteq [0, \infty)\) be an index set and \((\Omega^*, \mathscr{F}^*_t)_{t \in \mathbb{T}}\) be a sequence of measurable spaces.
\begin{definition}\label{def: SS}
We call \((\Omega^*, \mathscr{F}^*_t)_{t \in \mathbb{T}}\) a standard system, if 
\begin{enumerate}
\item[\textup{(i)}]
\(\mathscr{F}^*_t \subseteq \mathscr{F}^*_s\) for \(t, s \in \mathbb{T}\) with \(t < s\), 
\item[\textup{(ii)}]
for each \(t \in \mathbb{T}\) the space \((\Omega^*, \mathscr{F}^*_t)\) is a standard Borel space, i.e. \(\mathscr{F}^*_t\) is \(\sigma\)-isomorphic to the Borel \(\sigma\)-field of a Polish space,
\item[\textup{(iii)}] for each increasing sequence \((t_n)_{n \in \mathbb{N}}\) of elements in \(\mathbb{T}\) and any decreasing sequence \((A_n)_{n \in \mathbb{N}}\), where \(A_n\) is an atom in \(\mathscr{F}^*_{t_n}\), we have \(\bigcap_{n \in \mathbb{N}} A_n \not = \emptyset\). 
\end{enumerate}
\end{definition}
It is shown in \cite[Appendix]{follmer72} that
\(\left(\Omega, \mathscr{F}_{\sigma_n -}\right)_{n \in \mathbb{N}}\) is a standard system. Here, the choice of the underlying measurable space is crucial, because \((\Omega, \mathscr{F}^o_t)_{t \geq 0}\) is a standard system, too. Furthermore, it is important to define \(\mathscr{F}_{\sigma_n-}\) with \(\mathscr{F}^o_0\) instead of \(\mathscr{F}_0\), because the proof of Definition \ref{def: SS} (ii) requires \(\mathscr{F}_{\sigma_n-}\) to be countably generated.
For \(n \in \mathbb{N}\) define the probability measure \[\q_{n} \triangleq Z_{\sigma_n} \cdot \p\] on \((\Omega, \mathscr{F})\).
We deduce from Parthasarathy's extension theorem, see \cite[Theorem V.4.2]{parthasarathy1967}, together with \cite[Theorem D.4, Lemma E.1]{perkowski2015}, where it is again important that \(\mathscr{F}_{\sigma-}\) is countably generated, that there exists a probability measure \(\q\) on \((\Omega,\mathscr{F})\) such that
\begin{align*}
\q =\q_n 
\text{ on } \mathscr{F}_{\sigma_n-} \textup{ for all } n \in \mathbb{N}.
\end{align*} 

Next, we show that this equality even holds on the larger \(\sigma\)-field \(\mathscr{F}_{\sigma_n}\).
For all \(G \in \mathscr{F}_{\sigma_n}\), we have \(G \cap \{\sigma_n < \sigma_m\} \in \mathscr{F}_{\sigma_n} \cap \mathscr{F}_{\sigma_m-}\), which yields 
\begin{align*}
Q(G \cap \{\sigma_n < \sigma\}) &= \lim_{m \to \infty} Q(G \cap \{\sigma_n < \sigma_m\}) \\&= \lim_{m \to \infty} E^P\left[ Z_{\sigma_m} \1_{G \cap \{\sigma_n< \sigma_m\}}\right] \\&= \lim_{m \to \infty} E^P\left[ Z_{\sigma_n} \1_{G \cap \{\sigma_n< \sigma_m\}}\right] \\&= E^P\left[Z_{\sigma_n} \1_{G \cap \{\sigma_n < \sigma\}}\right] \\&= E^P\left[Z_{\sigma_n} \1_G\right]\\&= Q_n(G),
\end{align*}
due to the monotone convergence theorem, the optional stopping theorem and our assumption that \(P\)-a.s. \(\sigma_n < \sigma\). In particular, \(Q (\sigma_n < \sigma) = Q_n(\Omega) = 1\).
Thus, we have shown 
\begin{align*}
Q(G) = Q(G \cap \{\sigma_n < \sigma\}) = Q_n(G),
\end{align*}
i.e. in other words
\begin{align*}
Q = Q_n \textup{ on }\mathscr{F}_{\sigma_n} \textup{ for all } n \in \mathbb{N}.
\end{align*}

Next, we show that \(\q\) solves the SMP \((\zeta; \eta; B', C, \nu')\). Set \(\zeta_n \triangleq \sigma_n \wedge \rho_n\).
Since \(\q_{n} \ll \p\) with density process 
\[Z_{\sigma_n \wedge t} = \frac{\dd \q_n}{\dd \p}\bigg|_{\mathscr{F}_{t}},\]
we deduce from Girsanov's theorem for semimartingales, see \cite[Theorem III.3.24]{JS}, that the stopped process \(X^n_{\cdot \wedge \zeta_n}\) is a \(\q_n\)-semimartingale whose characteristics are \(\q_n\)-indistinguishable from \((B'_{\cdot \wedge \zeta_n}, C_{\cdot \wedge \zeta_n}, \1_{\of 0, \zeta_n\gs} \cdot \nu')\). 
Here, we use that, up to a \(P\)-evanescence set (and since \(Q_n \ll P\) also up to a \(Q_n\)-evanescence set), \(X^{k, c} = X^{k+1, c}\) on \(\of 0, \rho_k\gs\) for all \(k \in \mathbb{N}\), which follows from the uniqueness of the continuous local martingale part. 

Let us now transfer this observation from \(Q_n\) to the extension \(Q\). 
We can consider \(X^n_{\cdot \wedge \zeta_n}\) as a semimartingale on the filtered probability space \((\Omega, \mathscr{F}_{\zeta_n}, (\mathscr{F}_{t \wedge \zeta_n})_{t \geq 0}, Q_n)\), see \cite[Section 10.1]{J79}. The identity \(Q = Q_n\) on \(\mathscr{F}_{\zeta_n}\subseteq \mathscr{F}_{\sigma_n}\) implies that \(X^n_{\cdot \wedge \zeta_n}\) is an  \(\q\)-semimartingale whose characteristics are \(\q\)-indistinguishable from \((B'_{\cdot \wedge \zeta_n}, C_{\cdot \wedge \zeta_n}, \1_{\of 0, \zeta_n\gs} \cdot \nu')\). 
We conclude that \(\q\) solves the SMP \((\zeta; \eta'; B', C, \nu')\).

We proceed with the proofs of (a) -- (d).
\begin{enumerate}
\item[\textup{(a)}]
Let \(\xi\) be a stopping time and \(A \in \mathscr{F}_{\xi}\), then 
\begin{align*}
\q(A \cap \{\sigma > \xi\}) &= \lim_{n \to \infty} \q(A \cap \{\sigma_n > \xi\}) 
\\&=
\lim_{n \to \infty} \E^\p\left[Z_{\sigma_n} \1_{A \cap\{\sigma_n > \xi\}}\right] 
\\&= \lim_{n \to \infty} \E^\p\left[Z_{\xi} \1_{A \cap \{\sigma_n > \xi\}}\right]
\\&= \E^\p\left[Z_{\xi} \1_{A \cap \{\sigma > \xi\}}\right],
\end{align*}
again due to the monotone convergence theorem and the optional stopping theorem. 

\item[\textup{(b)}]
Suppose that (b.i) holds. Then, due to (a), we obtain
\begin{align*}
1 = Q(\sigma > t) = E^P \left[Z_t \1_{\{\sigma > t\}}\right].
\end{align*}
Since \(Z\) is a \(P\)-supermartingale by Fatou's lemma, we have 
\begin{align*}
E^P[Z_t] \leq E^P[Z_0] = 1.
\end{align*}
We conclude that 
\begin{align*}
0 \leq E^P\left[Z_t \1_{\{t \geq \sigma\}}\right] = E^P[Z_t]  - 1 \leq 0,
\end{align*}
which implies that \(P\)-a.s. \(Z_t = 0\) on \(\{t \geq \sigma\}\). 
This yields that for all \(G \in \mathscr{F}_t\)
\begin{align*}
Q(G) = E^P \left[Z_t\1_G\right],
\end{align*}
and \(Q \ll_\textup{loc} P\) with \(Z_t = \frac{\dd Q}{\dd P}|_{\mathscr{F}_t}\) follows immediately. In particular, \(Z\) is a \(P\)-martingale. In other words, we have shown that (b.i) \(\Rightarrow\) (b.ii) and that (b.i) implies \(Q \ll_\textup{loc} P\) with \(Z_t = \frac{\dd Q}{\dd P}|_{\mathscr{F}_t}\).

It remains to prove the implication (b.ii) \(\Rightarrow\) (b.i). If (b.ii) holds, (a) implies that for all \(t \geq 0\)
\begin{align*}
Q(\sigma > t) = E^P\left[Z_t \1_{\{\sigma > t\}}\right] = E^P[Z_t] = E^P[Z_0] = 1.
\end{align*}
It follows that \(Q(\sigma = \infty) = 1\), i.e. that (b.i) holds.

\item[\textup{(c)}] To see the implication (c.ii) \(\Rightarrow\) (c.i), we set \(\gamma_n \triangleq \sigma\) for all \(n \in \mathbb{N}\).
Then, the implication (c.ii) \(\Rightarrow\) (b.ii) \(\Leftrightarrow\) (b.i) yields that this sequence has all properties as claimed in (c.i).

Let us assume that (c.i) holds. Since (c.i) \(\Rightarrow\) (b.i) \(\Leftrightarrow\) (b.ii), it suffices to prove that \(Z\) is a uniformly integrable \(P\)-martingale. In fact, since \(Z\) is a \(P\)-supermartingale, it suffices to show that \(E^P[Z_\infty] \geq 1\).
Let \(A \in \mathscr{F}_{\gamma_n} \cap \mathscr{F}_{\sigma_m} = \mathscr{F}_{\gamma_n \wedge \sigma_m}\). Then, 
\begin{align*}
Q(A) = E^P \left[Z_{\sigma_m} \1_A\right] = E^P\left[Z_{\gamma_n} \1_A\right],
\end{align*}
where we use \eqref{eq: eqality sigma} and the optional stopping theorem.
By a monotone class argument, we have 
\begin{align*}
Q = Z_{\gamma_n} \cdot P \textup{ on } \mathscr{F}_{\gamma_n -},
\end{align*}
where we use that \(\gamma_n \leq \sigma\).
Note that \(\{\gamma_n = \sigma = \infty\} \in \mathscr{F}_{\gamma_n -}\). 
Thus, we obtain  
\begin{align*}
1 = \lim_{n \to \infty} Q\left(\gamma_n = \sigma = \infty\right)
&= \lim_{n \to \infty} E^P \left[ Z_{\gamma_n} \1_{\{\gamma_n = \sigma = \infty\}}\right]
\\&= \lim_{n \to \infty} E^P \left[ Z_{\infty} \1_{\{\gamma_n = \sigma = \infty\}}\right]
\\&\leq E^P \left[Z_\infty\right].
\end{align*} 
This proves (c.i) \(\Rightarrow\) (c.ii).

Finally, if (c.ii) holds, then (b) implies that \(Q \ll_{\textup{loc}} P\) with \(\frac{\dd Q}{\dd P}|_{\mathscr{F}_t} = Z_t\). Hence, \(Q \ll P\) with \(\frac{\dd Q}{\dd P} = Z_\infty\) follows immediately from \cite[Proposition III.3.5]{JS} and the uniform integrability of \(Z\).
\item[\textup{(d)}] 
We first show that (d.ii) \(\Rightarrow\) (d.i).
If \(P\)-a.s. \(\rho_n < \sigma\) and \(E^P\big[Z_{\rho_n}\big] = 1\), then (a) yields 
\begin{align}\label{eq: Q implies by P}
Q(\rho_n < \sigma) &= E^P \big[Z_{\rho_n} \1_{\{\rho_n < \sigma\}}\big] = E^P \big[Z_{\rho_n}\big] = 1.
\end{align}
Thus, (d.ii) \(\Rightarrow\) (d.i). 

Suppose that (d.i) holds.
We define the sequence of stopping times (see, e.g., \cite[Theorem III.3.9]{HWY} for the fact that the following are stopping times)
\begin{align*}
\gamma_m \triangleq \begin{cases} \sigma_m,&\textup{on } \{\rho_n \geq \sigma_m\},\\\infty,&\textup{otherwise},
\end{cases}
\end{align*}
and note that \(Q\)-a.s. \(\gamma_m \nearrow \infty\) as \(m \to \infty\) and \(\rho_n \wedge \gamma_m = \rho_n \wedge \sigma_m\). 
As above, it follows from Girsanov's theorem that for all \(m \in \mathbb{N}\) the process \(X^n_{\cdot \wedge \rho_n \wedge \gamma_m}\) is a \(Q\)-semimartingale whose characteristics are \(Q\)-indistinguishable from the triplet \((B'_{\cdot \wedge \rho_n \wedge \gamma_m}, C_{\cdot \wedge \rho_n \wedge \gamma_m}, \1_{\of 0, \rho_n \wedge \gamma_m\gs} \cdot \nu')\). Recall the fact that the class of semimartingales is stable under localization, see \cite[Proposition I.4.25]{JS}. This fact implies that \(X^n_{\cdot \wedge \rho_n}\) is a \(Q\)-semimartingale whose characteristics are \(Q\)-indistinguishable from \((B'_{\cdot \wedge \rho_n}, C_{\cdot \wedge \rho_n}, \1_{\of 0, \rho_n\gs} \cdot \nu')\).
In other words, we have shown that \(Q\) solves the SMP \((\rho; \eta'; B', C, \nu')\). 
\end{enumerate}
The proof is complete.
\qed
\\\\
In the next section, we discuss consequences of Theorem \ref{theo:main1}.
\section{Absolute Continuity of Semimartingales}\label{sec:LE}
In this section we study absolute continuity of semimartingales. 
Systematic approaches in conservative settings were given by Kabanov, Lipster and Shiryaev \cite{KLS-LACOM1,KLS-LACOM2}, Jacod and M\'emin \cite{JM76} and Jacod \cite{J79, Jacod1979} under a strong uniqueness assumption, called local uniqueness in the monograph \cite{JS}.
As we show below, the local uniqueness assumption can be replaced by a usual uniqueness assumption. This is well-known to be true in Markovian settings and very surprising to hold in all generality.

%
%

Let \(\beta\colon \Omega \times [0, \infty) \to \mathbb{R}^d\) be predictable and \(U \colon \Omega \times [0, \infty) \times \mathbb{R}^d \to [0, \infty)\) be \(\widetilde{\mathscr{P}}\)-measurable. Furthermore, let \(B, C\) and \(\nu\) are given as in Section \ref{sec: main}.
Additionally, we suppose that \(C\) admits a decomposition
\[
C = \int_0^\cdot c_s \dd A_s,
\]
where \(c\) is a predictable \(\mathbb{S}^d\)-valued process and \(A\) is a non-negative, increasing, predictable and right-continuous process starting in the origin. Here, \(\mathbb{S}^d\) denotes the set of all symmetric non-negative definite real \(d \times d\) matrices. The integral is set to \(\Delta\) if it diverges.
We further set 
\begin{equation}\label{eq:candJ}
\begin{split}
B' &\triangleq B + \int_0^\cdot  c_s\beta_s \dd A_s + h(x) (U - 1) \star \nu,\\ \nu' &\triangleq U\cdot  \nu,
\end{split}
\end{equation}
where an integral is set to be \(\Delta\) whenever it diverges.
The first standing assumption in this section is the following.
\begin{SA}\label{SA 2}
	Let \(P\) be a solution to the SMP \((\rho ; \eta; B, C, \nu)\), with \(\rho\)-localizing sequence \((\rho_n)_{n \in \mathbb{N}}\) and fundamental sequence \((X^n)_{n \in \mathbb{N}}\), and let \(Q^*\) be a solution to the SMP \((\rho; \eta; B', C, \nu')\). W.l.o.g. \(\rho_n \leq n\). 
\end{SA}
For all \(t \geq 0\) we define
\[
\widehat{U}_t \triangleq \int U(t, x) \nu(\{t\} \times \dd x),\qquad a_t \triangleq \nu(\{t\} \times \mathbb{R}^d).
\]
\begin{SA} For all \(t \geq 0\) we have identically \(a_t \leq 1\) and \(\widehat{U}_t \leq 1\).
\end{SA} We set
\begin{equation}\label{eq:H}
\begin{split}
H_t \triangleq \int_0^{t \wedge \rho}  \langle \beta_s&, c_s \beta_s \rangle \dd A_s + \left(1 - \sqrt{U}\right)^2 \star \nu_{t \wedge \rho} 
+ \sum_{s \leq t \wedge \rho} \left(\sqrt{1 - a_s} - \sqrt{1 - \widehat{U}_s}\right)^2,
\end{split} 
\end{equation}
and define 
\begin{equation}\label{eq: sigma}
\begin{split}
\sigma_n &\triangleq \inf(t \geq 0 \colon H_t \geq n) \wedge n,\qquad \sigma \triangleq \lim_{n \to \infty} \sigma_n.
\end{split}
\end{equation}
The process \(H\) is increasing, but not in the sense of \cite{JS}, because it may fail to be right-continuous, i.e.
on \(\{\sigma < \infty\}\) it can happen that \(H_\sigma < \infty\) and \(H_{\sigma +} = \infty\). Here, we stress that increasing functions have left and right limits. 

In the following standing assumption we suppose that \(H\) is right-continuous and that \(H\) can only attend \(\infty\) in a continuous manner. 
We comment on this in Remark \ref{rem: jump inf} below.
\begin{SA}\label{SA 3}
Up to \(P\)-evanescence, on \(\bigcup_{n \in \mathbb{N}} \of 0, \sigma_n \wedge \rho_n\gs\) 
\begin{equation}\label{cond:Ya}
\begin{split}
a = 1\ &\Rightarrow\ \widehat{U} = 1.
\end{split}
	\end{equation}
Furthermore, one of the following holds:
\begin{enumerate}
	\item[\textup{(a)}] \(P\)-a.s. \(\rho_n < \sigma\) for all \(n \in \mathbb{N}\).
	\item[\textup{(b)}] \(P\)-a.s. \(H_\sigma = \infty\) on \(\{\sigma < \infty\}\), and for all solutions \(\widehat{Q}\) to the SMP \((\sigma \wedge \rho; \eta; B', C, \nu')\) we have \(\widehat{Q}\)-a.s. \(\rho_n < \sigma\) for all \(n \in \mathbb{N}\).
\end{enumerate}
\end{SA}

To get an intuition for the condition \eqref{cond:Ya}, suppose that \(P\) and \(Q^*\) are laws of \(\mathbb{R}^d\)-valued semimartingales with independent increments. In this case, the triplets \((B, C, \nu)\) and \((B', C, \nu')\) are deterministic and we have
\begin{align*}
P(\Delta X_t \in \dd x) &= \1_{\mathbb{R}^d \backslash \{0\}} \nu(\{t\} \times \dd x) + \left(1 - a_t\right) \delta_0 (\dd x),\\
Q^*(\Delta X_t \in \dd x) &= \1_{\mathbb{R}^d \backslash \{0\}} U(t, x)\nu(\{t\} \times \dd x) + \left(1 - \widehat{U}_t\right) \delta_0 (\dd x),
\end{align*}
where \(\delta\) denotes the Dirac measure, see \cite[Theorem II.4.15]{JS}.
If \(a_t = 1\), then \(Q^* (\Delta X_t \in \dd x) \ll P (\Delta X_t \in \dd x)\) can only be true when \(\widehat{U}_t = 1\). The absolute continuity \(Q^* (\Delta X_t \in \dd x) \ll P (\Delta X_t \in \dd x)\) is implied by \(Q^* \ll_{\textup{loc}} P\) and therefore \eqref{cond:Ya} is very natural.


If \(P\)-a.s. \(\rho_n < \sigma\) for all \(n \in \mathbb{N}\), then \(P\)-a.s.
\begin{align}\label{eq: sigma decomp}
\sigma = \begin{cases}
\rho,& \textup{if } H_{\rho} = \infty \textup{ on } \{\rho < \infty\},\\
\infty,&\textup{otherwise},
\end{cases}
\end{align}
which yields that \(P\)-a.s. \(H_{\sigma} = \infty\) on \(\{\sigma < \infty\}\), i.e. \(H\) is right-continuous.
We compute that
\begin{equation}\label{eq: H bounded}
\begin{split}
H_{\sigma_n} &\leq n + \Delta H_{\sigma_n} 
\\&\leq n + 2 \left(a_{\sigma_n} + \widehat{U}_{\sigma_n} + 1 - a_{\sigma_n} + 1 - \widehat{U}_{\sigma_n} \right)
\\&= n + 4.
\end{split}
\end{equation}
Thus, if \(P\)-a.s. \(H_{\sigma} = \infty\) on \(\{\sigma < \infty\}\), we have \(P\)-a.s. \(\sigma_n < \sigma\). We note that this inequality also holds on \(\{\sigma = \infty\}\) since \(\sigma_n\) is bounded by \(n\), see \eqref{eq: sigma}. Providing an intuition, this means that \(H\) cannot jump to \(\infty\).

Next, we define a non-negative local martingale \(Z\) which relates \(P\) and \(Q^*\). 
We find a non-negative local \(P\)-martingale on \(\bigcup_{n \in \mathbb{N}} \of 0, \sigma_n \wedge \rho_n\gs\) which coincides with the stochastic exponential of
\begin{align}\label{eq: N}
\int_0^{\cdot} \la \beta_s, \dd X^{n, c}_s\ra + \left(U - 1 + \frac{\widehat{U} - a}{1 - a}\right)  \star (\mu^{X^n} - \nu)
\end{align}
on the random set \(\of 0, \sigma_n \wedge \rho_n\gs\), see \cite[Proposition II.1.16]{JS}. Here, we use the convention that \(\tfrac{0}{0} \equiv 0\). The second stochastic integral denotes the discontinuous local \(P\)-martingale whose jump process is \(P\)-indistinguishable from 
\begin{align}\label{eq: delta N}
(U(\cdot, \Delta X^n) - 1) \1_{\{\Delta X^n \not = 0\}} - \left(\frac{\widehat{U} - a}{1 - a} \right)\1_{\{\Delta X^n = 0\}}, 
\end{align}
see \cite[Section II.1]{JS} for more details. The non-negativity follows from the fact that \eqref{cond:Ya} implies that \eqref{eq: delta N} is, up to \(P\)-evanescence, greater or equal than \(-1\) on \(\bigcup_{n \in \mathbb{N}} \of 0, \sigma_n \wedge \rho_n\gs\), see \cite[Theorem I.4.61]{JS}.
As pointed out in Remark \ref{rem: extension}, we can extend this non-negative local \(P\)-martingale on \(\bigcup_{n \in \mathbb{N}} \of 0, \sigma_n \wedge \rho_n\gs\) to a global one, which we denote \(Z\) in the following.

By \cite[Theorem 8.25]{J79} and similar arguments as used in the proof of \cite[Lemma 12.44]{J79}, \eqref{eq: H bounded} implies that the stopped process \(Z^{\sigma_n}\) is a uniformly integrable \(P\)-martingale and, since \(P\)-a.s. \(H_{\sigma} = \infty\) on \(\{\sigma < \infty\}\), \cite[Theorem 8.10]{J79} yields that \(P\)-a.s. \(Z = 0\) on \(\of \sigma, \infty\of\).

\begin{remark}\label{rem: jump inf}
If \(H\) is allowed to jump to \(\infty\), it may happen that \(Z\) is positive on \(\of \sigma, \infty\of\) with positive \(P\)-probability. In this case, the parts (b) and (c) in Theorem \ref{theo:main1} are empty. 
Of course, we could modify \(Z\) to be zero on \(\of \sigma, \infty\of\), but then the modification might only be a supermartingale.
Let us discuss an explicit example.
	Consider a one-dimensional \([- \infty, \infty]\)-valued diffusion 
	\begin{align*}
	\dd Y_t = \mu(Y_t) \dd t + a(Y_t)\dd W_t, 
	\end{align*}
	where \(W\) is a one-dimensional Brownian motion. If \(\mu\) and \(a\) satisfy the Engelbert-Schmidt conditions, see \cite{KaraShre} or Standing Assumption \ref{SA 4} below, then \(Y\) exists up to an explosion time \(\theta\).
	In this case, for any Borel function \(f\colon \mathbb{R} \to [0, \infty)\), the integral process
	\begin{align*}
	K \triangleq \int_0^{\cdot \wedge \theta} f(Y_s)\dd s
	\end{align*}
	is similar to \(H\). Let \(D\) be the set of all \(x \in \mathbb{R}\) for which there is no \(\epsilon > 0\) such that
	\begin{align*}
	\int_{x - \epsilon}^{x + \epsilon} \frac{f(y)}{a^2(y)} \dd y < \infty,
	\end{align*}
	and denote 
	\begin{align*}
	\eta_D \triangleq \theta \wedge \inf(t \geq 0 \colon Y_t \in D).
	\end{align*}
	By \cite[Theorem 2.6]{mijatovic2012}, we have a.s.
	\begin{align*}
	K_t\  \begin{cases} < \infty,& t \in [0, \eta_D),\\ = \infty,& t \in (\eta_D, \theta].\end{cases}
	\end{align*}
	This characterization follows naturally from the occupation times formula, which states that for \(t < \theta\) a.s.
	\[
	\int_0^t f(Y_s)\dd s = \int_0^t \frac{f(Y_s)}{a^2(Y_s)}\dd\hspace{0.05cm} \lle Y, Y\rre_s = \int_{- \infty}^\infty \frac{f(y)}{a^2(y)} L_t^y(Y) \dd y,
	\]
	where \(L\) denotes the local time, see \cite[Equation (4.4)]{MANA:MANA19911510111}.
	On the set \(\{\eta_D < \theta\}\) it might happen with positive probability that \(K_{\eta_D} < \infty\), see \cite[Sections 2.4, 2.5]{mijatovic2012} for more details. In this case, \(K\) jumps to infinity and the extension \(Z\) is positive on \(\of \eta_D, \infty\of\) with positive probability. Deterministic conditions for this case can be found in \cite{mijatovic2012}.
	Finally, we stress that a.a. paths of \(K\) do not jump to infinity if \(D = \emptyset\).
\end{remark}
\begin{SA}\label{SA3a}
	Standing Assumption \ref{SA 3} holds with \textup{(a)} replaced by 
	\begin{enumerate}
		\item[\textup{(a)'}] \(P\)-a.s. \(\rho_n < \sigma\) and \(E^P \big[Z_{\rho_n}\big] = 1\) for all \(n \in \mathbb{N}\).
	\end{enumerate}
\end{SA}
The additional assumption has a local character. In fact, in many cases it follows easily from classical moment conditions such as Novikov-type conditions. For the readers convenience we collect two conditions:
\begin{proposition}\label{prop: Nov}
	Let \(n \in \mathbb{N}\).
	Assume that at least one of the following conditions holds:
	\begin{enumerate}
		\item[\textup{(i)}] The random variable \(H_{\rho_n}\) is bounded up to a \(P\)-null set. 
		\item[\textup{(ii)}] Set 
		\begin{align*}
		H^* \triangleq \frac{1}{2} \int_0^{\cdot} \langle \beta_s, c_s \beta_s\rangle \dd A_s 
		&+ \sum_{s \leq \cdot} \left(\big(1 - \widehat{U}_s\big) \log \left(\frac{1 - \widehat{U}_s}{1 - a_s}\right) + \widehat{U}_s - a_s \right)\\&+ \left(U \log(U) - U + 1 \right)\star \nu_{\cdot},
		\end{align*}
		where we use the convention that \(0 \times (- \infty) \equiv 0\).
		It holds that \(E^P \big[\exp (H^*_{\rho_n})\big] < \infty\).
	\end{enumerate}
	Then, \(E^P\big[Z_{\rho_n}\big] = 1\).
\end{proposition}
\begin{proof}
	The identity \(E^P\big[Z_{\rho_n}\big] = 1\) is implied by (i) due to similar arguments as used in the proof of \cite[Lemma 12.44]{J79} together with \cite[Theorem 8.25]{J79}. 
	Furthermore, \(E^P\big[Z_{\rho_n}\big] = 1\) is implied by (ii) due to \cite[Corollary 8.44]{J79}. 
\end{proof}
Next, we state the main result of this section.


\begin{corollary}\label{coro:LE}
Assume that all solutions to the SMP \((\rho; \eta; B', C, \nu')\) coincide on the \(\sigma\)-field \(\mathscr{F}_{\sigma-}\). Then, for all stopping times \(\xi\) we have 
\begin{align}\label{eq: CMG2}
Q^* = Z_\xi \cdot P \text{ on } \mathscr{F}_{\xi} \cap \{\sigma > \xi\}.
\end{align}
Moreover, we have the following:
\begin{enumerate}
	\item[\textup{(a)}]
The following are equivalent:
\begin{enumerate}
	\item[\textup{(a.i)}]
\(Q^*\)-a.s. \(H_t < \infty\) for all \(t \geq 0\).
\item[\textup{(a.ii)}]
The process \(Z\) is a \(P\)-martingale.
\item[\textup{(a.iii)}] \(Q^* \ll_{\textup{loc}} P\) with \(\frac{\dd Q^*}{\dd P} |_{\mathscr{F}_t} = Z_t\).
\end{enumerate}
\item[\textup{(b)}]
The following are equivalent:
\begin{enumerate}
	\item[\textup{(b.i)}]
\(Q^*\)-a.s. \(H_\rho < \infty\).
\item[\textup{(b.ii)}]
The process \(Z\) is a uniformly integrable \(P\)-martingale.
\item[\textup{(b.iii)}] \(Q^* \ll P\) with \(\frac{\dd Q^*}{\dd P} = Z_\infty\).
\end{enumerate}
\end{enumerate}
\end{corollary}
\begin{proof}
First, note that Standing Assumption \ref{SA1} holds. Let \(Q\) be as in Theorem \ref{theo:main1}, which yields, due to Standing Assumption \ref{SA3a}, that \(Q\) solves the SMP \((\rho; \eta; B', C, \nu')\). Furthermore, by hypothesis, \(Q\) coincides with \(Q^*\) on \(\mathscr{F}_{\sigma-}\). Hence, the formula \eqref{eq: CMG2} immediately follows from Theorem \ref{theo:main1} (a).

Since for all \(G \in \mathscr{F}\) we have \(G \cap \{\sigma = \infty\} \in \mathscr{F}_{\sigma-}\), the equivalence (a.i) \(\Leftrightarrow\) (a.ii) follows from Theorem \ref{theo:main1} (b).
If (a.i) holds, then \(Q = Q^*\) on \(\mathscr{F}\). We explain this with more details. Since \(Q = Q^*\) on \(\mathscr{F}_{\sigma-}\) and \(\{\sigma=\infty\}\in \mathscr{F}_{\sigma-}\), we have \(Q\)-a.s. \(\sigma = \infty\). Now, for all \(G \in \mathscr{F}\), we have \(G \cap \{\sigma = \infty\} \in \mathscr{F}_{\sigma-}\) and therefore
\[
Q(G) = Q(G \cap \{\sigma= \infty\}) = Q^*(G \cap \{\sigma= \infty\}) = Q^*(G).
\]
Consequently, (a.i) \(\Rightarrow\) (a.iii) follows from Theorem \ref{theo:main1} (b), too.
Since the implication (a.iii) \(\Rightarrow\) (a.ii) is trivial, this completes the proof of (a).

Set 
\begin{align*}
\gamma_n \triangleq \inf\left(t \geq 0\colon H_t \geq n\right), 
\end{align*}
and note that (b.i) implies that 
\begin{align*}
\lim_{n \to \infty} Q^* (\gamma_n = \sigma = \infty) = 1.
\end{align*}
Furthermore, by \cite[Theorem 8.25]{J79} and similar arguments as used in the proof of \cite[Lemma 12.44]{J79}, the stopped process \(Z^{\gamma_n}\) is a uniformly integrable \(P\)-martingale.
Thus, the implication (b.i) \(\Rightarrow\) (b.ii) follows from Theorem \ref{theo:main1} (c).

If (b.ii) holds, then (a.ii) and thus also (a.i) holds and we have \(Q = Q^*\) on \(\mathscr{F}\). 
Hence, the implication (b.ii) \(\Rightarrow\) (b.iii) is due to Theorem \ref{theo:main1} (c).

Finally, the implication (b.iii) \(\Rightarrow\) (b.i) follows from \cite[Theorem 8.21, Lemma 12.44]{J79} and the proof is complete.
\end{proof}
\begin{remark}\label{rem: JS}
	Recalling the equalities \eqref{eq: Q implies by P} and \eqref{eq: sigma decomp}, if (a)' in Standing Assumption \ref{SA3a} holds, then (a.i) in Corollary \ref{coro:LE} is equivalent to \(Q^*\)-a.s. \(H_\rho < \infty\) on \(\{\rho < \infty\}\).
In this case, the difference between local absolute continuity and absolute continuity is captured by the behavior of \(H_{\rho}\) on the set \(\{\rho = \infty\}\).
\end{remark}
\begin{remark}\label{rem: uni}
	In many cases, for instance due to parametric constraints, all solutions to a SMP are supported on a path space \(\Omega^o \subseteq \Omega\), see \cite[Section 1.11]{pinsky1995positive} for examples. 
	In particular, this is the case when \(\rho = \infty\) with \(\Omega^o\) being the classical Skorokhod space, i.e. the space of all \cadlag functions \([0, \infty) \to \mathbb{R}^d\).
	In such a situation, uniqueness on \(\mathscr{F}_{\sigma-}\) is equivalent to uniqueness on the trace \(\sigma\)-field \(\mathscr{F}_{\sigma-} \cap \Omega^o\). If in addition \(\rho = \tau_\Delta\) on \(\Omega^o\), then 
	\begin{align}\label{eq: uni}\mathscr{F}_{\rho-} \cap \Omega^o = \mathscr{F}_{\tau_{\Delta}-} \cap \Omega^o = \mathscr{F} \cap \Omega^o.\end{align} Here, we use the identity \(\mathscr{F}_{\tau_\Delta-} = \mathscr{F}\), which follows from the following: For all \(G\in \mathscr{B}(\mathbb{R}^d_\Delta)\)
	\begin{align*}
	\{X_t \in G\} &= \left(\{X_t \in G\} \cap \{\tau_{\Delta} \leq t\}\right) \cup \left(\{X_t \in G\} \cap \{\tau_{\Delta} > t\}\right)
	\\&= \begin{cases}\{\tau_{\Delta} \leq t\} \cup \left(\{X_t \in G\} \cap \{\tau_{\Delta} > t\}\right),&\Delta \in G,\\
	\{X_t \in G\} \cap \{\tau_{\Delta} > t\},&\Delta \not \in G.
	\end{cases}
	\end{align*}
	The sets on the right hand side are in \(\mathscr{F}_{\tau_{\Delta}-}\). Thus, we have shown the inclusion \(\mathscr{F} \subseteq \mathscr{F}_{\tau_{\Delta}-}\), which implies the identity \(\mathscr{F}_{\tau_{\Delta}-} = \mathscr{F}\).
	Coming back to the identity \eqref{eq: uni}, we see that uniqueness on \(\mathscr{F}_{\rho-} \cap \Omega^o\) implies uniqueness on \(\mathscr{F}\) and in particular uniqueness on \(\mathscr{F}_{\sigma-}\).
\end{remark}
\begin{corollary}\label{coro: JS}
	Suppose that \(\rho = \infty\), that \(Q^*\) is the only solution of the SMP \((\rho; \eta; B', C, \nu')\) and that \textup{(a)} in Standing Assumption \ref{SA 3} holds. Then, \(Q^* \ll_\textup{loc} P\) with \(\frac{\dd Q^*}{\dd P}|_{\mathscr{F}_t} = Z_t\) and the following are equivalent:
	\begin{enumerate}
		\item[\textup{(i)}] \(P \ll_{\textup{loc}} Q^*\)
		\item[\textup{(ii)}] \(P\)-a.s. \(\1_{\{U = 0\}} \star \nu_{\infty} = 0\) and \(\widehat{U} = 1\Rightarrow a = 1\).
	\end{enumerate}
\end{corollary}
\begin{proof} This follows from Corollary \ref{coro:LE}, Remark \ref{rem: JS} and  \cite[Theorem 12.48]{J79}.\end{proof}

This result shows that the \(Q^*\)-integrability condition and the local uniqueness conditions imposed in \cite[Theorem III.5.34]{JS} can be replaced by a simple uniqueness condition together with a \(P\)-integrability condition.

Another consequence of Corollary \ref{coro:LE} is that if \(Q^*\) is unique and \(H\) is finite and deterministic, then \(Q^* \ll_{\textup{loc}} P\) with \(\tfrac{\dd Q^*}{\dd P}|_{\mathscr{F}_t}= Z_t\). This observation can be proven directly with the same strategy as used in the proof of Theorem \ref{theo:main1}:
Indeed, if \(H\) is finite and deterministic, the local \(P\)-martingale \(Z\) has a deterministic \(P\)-localizing sequence, namely \((\sigma_n)_{n \in \mathbb{N}}\). Consequently, \(Z\) is a true \(P\)-martingale. Now, applying Parathasaraty's extension theorem to the standard system \((\Omega, \mathscr{F}^o_n)_{n \in \mathbb{N}}\), see \cite{follmer72}, \(Q^*\) can be constructed from \(P\) as an extension of the consistent sequence \((Z_n \cdot P)_{n \in \mathbb{N}}\). By construction, \(Q^* \ll_{\textup{loc}} P\) with \(\tfrac{\dd Q^*}{\dd P}|_{\mathscr{F}_t} = Z_t\).


In the following section we comment on related literature. We will also see further applications of Corollary \ref{coro:LE}.
\section{Comments on the Literature}\label{sec: comment}
In this section we relate our results to the literature. In Section \ref{sec: Diff}, we show that Corollary \ref{coro:LE} is in line with the main results of Cherny and Urusov \cite{Cherny2006} and Mijatovi\'c and Urusov \cite{MU(2012)}. In Section \ref{sec: CFY}, we relate Corollary \ref{coro:LE} to the main result of Cheridito, Filipovi\'c and Yor \cite{CFY}.

\subsection{Absolute Continuity of One-Dimensional Diffusions}\label{sec: Diff}
The (local) absolute continuity of laws of one-dimensional diffusions was intensively studied by Cherny and Urusov \cite{Cherny2006}, who gave deterministic equivalent conditions under the Engelbert-Schmidt conditions. In the same setting, deterministic equivalent conditions for the martingale property of stochastic exponentials were given by Mijatovi\'c and Urusov \cite{MU(2012)}. 
Both approaches are based on so-called separation times and are quite different from ours. As we will illustrate in this section, their results can also be deduced from Corollary \ref{coro:LE}. 

We start with a formal introduction to the setup.
In the following, \(\nu\) will always be the zero measure and we will remove it from all notations. 

Let \(b,\beta\colon\mathbb{R} \to \mathbb{R}\) and \(c\colon \mathbb{R} \to [0, \infty)\) be Borel functions. We extend these functions to \(\mathbb{R}_\Delta\) by setting them to zero outside \(\mathbb{R}\).
Furthermore, we set 
\begin{align*}
B &\triangleq \int_0^\cdot b (X_s) \dd s,\\ B' &\triangleq \int_0^\cdot \left(b(X_s) + (\beta c)(X_s)\right) \dd s,\\ C &\triangleq \int_0^\cdot c(X_s)\dd s,
\end{align*}
and define the stopping times
\begin{align}\label{eq: exit}
\rho_n \triangleq \inf(t \geq 0 \colon \|X_t\| > n) \wedge n,\quad \rho \triangleq \lim_{n \to \infty} \rho_n,
\end{align}
where \(\|\Delta\|\triangleq \infty\).
\begin{SAS}\label{SA 4}
The Engelbert-Schmidt conditions hold, i.e.
\begin{align*}
c> 0,\qquad \frac{1 + |b| + |b + \beta c|}{c} \in L^1_\textup{loc} (\mathbb{R}).
\end{align*}
\end{SAS}
In this case, for all \(x \in \mathbb{R}\) the SMP \((\rho; \delta_x; B, C)\) has a solution \(P\) with \(\rho\)-localizing sequence \((\rho_n)_{n \in \mathbb{N}}\) and the SMP \((\rho; \delta_x; B', C)\) has a solution \(Q^*\). 
Let \(\Omega^o\) be the set of all \(\omega \in \Omega\) which are continuous on \([0, \tau_{\Delta}(\omega))\) and \(\lim_{t \nearrow \tau_{\Delta} (\omega)} \omega(t) = - \infty\) or \(\lim_{t \nearrow \tau_{\Delta} (\omega)} \omega(t) = + \infty\) whenever \(\tau_{\Delta}(\omega) \in (0, \infty)\).
All solutions to each of these SMPs are supported on the set \(\Omega^o\) and coincide on \(\mathscr{F}_{\rho-} \cap \Omega^o\). Thus, all solutions to these SMPs coincide on \(\mathscr{F}\), see Remark \ref{rem: uni}. 
In particular, we have \(P\)-a.s. and \(Q^*\)-a.s. \begin{align}\label{eq: tau strict}
\rho_n < \rho.\end{align}
In other words, \(P, Q^*\) and \((\rho_n)_{n \in \mathbb{N}}\) are as in Standing Assumption \ref{SA 2} and the uniqueness assumption of Corollary \ref{coro:LE} holds.
Proofs are given in \cite[Section 5.5]{KaraShre} or \cite{MANA:MANA19911510111}. 

\begin{SAS}\label{SA 5} We have 
\begin{align}\label{eq: finite qv}
\beta^2 \in L^1_\textup{loc}(\mathbb{R}).
\end{align}
\end{SAS}
Standing Assumption \ref{SA 5} is also imposed in \cite{MU(2012)}, but not in \cite{Cherny2006}, where it is shown to be necessary for \(Q^* \ll_{\textup{loc}} P\).

The condition \eqref{eq: finite qv} implies that
\begin{align}\label{eq: finiteness inte}
\int_0^t (\beta^2 c)(X_s) \dd s < \infty,\quad P\text{-a.s. and } Q^*\text{-a.s. for all } t < \rho,
\end{align}
see \cite[Theorem 2.6]{mijatovic2012} and Remark \ref{rem: jump inf}.

Denote
\begin{align*}
H \triangleq \int_0^{\cdot \wedge \rho} (\beta^2 c)(X_s)\dd s 
\end{align*}
and set \(\sigma_n\) and \(\sigma\) as in \eqref{eq: sigma}. We note that \eqref{eq: tau strict} and \eqref{eq: finiteness inte} imply \(P\)-a.s. 
\begin{align}\label{eq:tau sigma ineq}
\rho_n < \rho \leq \sigma.
\end{align}
Thus, Standing Assumption \ref{SA 3} (a) holds.
We define a non-negative local \(P\)-martingale \(Z\) as in Section \ref{sec:LE}.

In this setting Standing Assumption \ref{SA3a} holds, too.
If the function \(\beta^2 c\) is locally bounded, \(E^P\big[Z_{\rho_n}\big] = 1\) follows immediately from Novikov's condition, see also Proposition \ref{prop: Nov}. However, under the weaker assumption that \(\beta^2 \in L^1_\textup{loc}(\mathbb{R})\) the verification becomes more challenging. We refer to \cite[Lemma 5.30]{Cherny2006} for a proof.

Consequently, the following result follows from Corollary \ref{coro:LE} and Remark \ref{rem: JS}.

\begin{corollary}\label{coro: diff}
	\begin{enumerate}
		\item[\textup{(a)}]
	The following are equivalent:\begin{enumerate}
		\item[\textup{(a.i)}]
		\(Q^*\)-a.s. \(H_{\rho} = \int_0^{\rho} (\beta^2 c)(X_s)\dd s <\infty\) on \(\{\rho < \infty\}\).
		\item[\textup{(a.ii)}]
		\(Z\) is a \(P\)-martingale.
		\item[\textup{(a.iii)}]
		\(Q^* \ll_{\textup{loc}} P\) with \(\frac{\dd Q^*}{\dd P}|_{\mathscr{F}_{t}} = Z_t.\)
	\end{enumerate}
		\item[\textup{(b)}]
The following are equivalent:\begin{enumerate}
	\item[\textup{(b.i)}]
	\(Q^*\)-a.s. \(H_{\rho} = \int_0^{\rho} (\beta^2 c)(X_s)\dd s <\infty\).
	\item[\textup{(b.ii)}]
	\(Z\) is a uniformly integrable \(P\)-martingale.
	\item[\textup{(b.iii)}]
	\(Q^* \ll P\) with \(\frac{\dd Q^*}{\dd P} = Z_\infty.\)
\end{enumerate}
\end{enumerate}
\end{corollary}
	A relation of the convergence of an integral functional and the martingale property of a stochastic exponential is also suggested in \cite[Section 2.6]{mijatovic2012}. Corollary \ref{coro: diff} confirms a one-to-one relation. 

Let us now explain that this corollary is in line with the deterministic conditions for the (local) absolute continuity as given in \cite{Cherny2006} and the (uniformly integrable) \(P\)-martingale property of \(Z\) as given in \cite{MU(2012)}. We start with notation:
\begin{equation}\label{eq: diff definitions}
\begin{split}
p (x) &\triangleq \exp \left( - \int_0^x \frac{2 (b(y) + (\beta c)(y))}{c(y)} \dd y \right), \ x \in \mathbb{R}, \\
	s(x) &\triangleq \int_0^x p(y)\dd y, \ x \in \mathbb{R},\\
	s(+ \infty) &\triangleq \lim_{x \nearrow + \infty} s(x),\\
	s(- \infty) &\triangleq \lim_{x \searrow - \infty} s(x).
\end{split}
\end{equation}
Furthermore, for \(z \in \{- \infty, \infty\}\) and a Borel function \(f\colon \mathbb{R} \to [0, \infty)\) we write \(f \in L^1_\textup{loc}(z)\) if there is an \(x \in \mathbb{R}\) such that \(\int_{x \wedge z}^{x \vee z} f(y)\dd y < \infty\). 
We define the following conditions:
\begin{gather}
s(+ \infty) = \infty,\label{cond +.1}\\
s(+ \infty) < \infty\quad \textup{and}\quad \frac{s (+ \infty) - s}{p c} \not\in L^1_\textup{loc} (\infty),\label{cond +.2}\\
s(+ \infty) < \infty \quad\textup{and} \quad \frac{(s(+ \infty) - s)\beta^2}{p} \in L^1_\textup{loc}(- \infty),\label{cond +.3}
\end{gather}
and similarly 
\begin{gather}
s(- \infty) = -\infty,\label{cond -.1}\\
s(- \infty) >- \infty\quad \textup{and}\quad \frac{s - s(- \infty)}{pc} \not\in L^1_\textup{loc} (-\infty),\label{cond -.2}\\
s(- \infty) >- \infty \quad\textup{and} \quad \frac{(s- s(- \infty))\beta^2}{p} \in L^1_\textup{loc}(- \infty).\label{cond -.3}
\end{gather}
Let us relate these conditions to (a.i) in Corollary \ref{coro: diff}. 
Define 
\begin{align*}
\rho_{+ } &\triangleq \lim_{n \to + \infty} \inf(t \geq 0 \colon X_t > n),\\
\rho_{- } & \triangleq \lim_{n \to + \infty} \inf(t \geq 0 \colon X_t < - n),
\end{align*}
and note that \(Q\)-a.s. \begin{align*}
\{\rho < \infty\} = \{\rho_{+ } < \infty\} \cup \{\rho_{- } < \infty\}.\end{align*}
We discuss the finiteness of \(H_{\rho}\) separately on the two sets on the right hand side.
By Feller's test for explosion, see \cite[Propositions 2.4, 2.5, 2.12]{mijatovic2012}, \(\{\rho_{+} < \infty\}\) is \(Q^*\)-null if and only if either \eqref{cond +.1} or \eqref{cond +.2} holds. 

If \(Q^*(\rho_{+} < \infty) > 0\), then \(H_{\rho}\) is \(Q^*\)-a.s. finite on \(\{\rho_{+ } < \infty\}\) 
if and only if \eqref{cond +.3} holds, see \cite[Theorem 2.11]{mijatovic2012}.

Similar arguments yield that \(H_{\rho}\) is \(Q^*\)-a.s. finite on \(\{\rho_{-} < \infty\}\) if and only if one of the conditions \eqref{cond -.1}, \eqref{cond -.2} or \eqref{cond -.3} holds. Finally, we recover the following version of \cite[Theorem 2.1]{MU(2012)} and \cite[Corollary 5.2]{Cherny2006} from Corollary \ref{coro: diff}.
\begin{corollary}
	\textup{(a.i), (a.ii)} and \textup{(a.iii)} from Corollary \ref{coro: diff} are equivalent to the following:
	\begin{enumerate}
		\item[\textup{(a.iv)}] One of the conditions \eqref{cond +.1}, \eqref{cond +.2} or \eqref{cond +.3} holds and one of the conditions \eqref{cond -.1}, \eqref{cond -.2} or \eqref{cond -.3} holds.
	\end{enumerate} 
\end{corollary}
Let us now explain when \(H_{\rho}\) is \(Q^*\)-a.s. finite everywhere. We distinguish four cases:
\begin{enumerate}
	\item[1.] If \(s(+ \infty) = \infty\) and \(s(- \infty) = \infty\), then \(Q^*\)-a.s. \(H_{\rho} < \infty\) if and only if Lebesgue almost everywhere \(\beta = 0\), see \cite[Theorem 2.10]{mijatovic2012}.
	\item[2.] If \(s(+ \infty) < \infty\) and \(s(- \infty) = \infty\), then \(Q^*\)-a.s. \(H_{\rho} < \infty\) if and only if the second part in \eqref{cond +.3} holds, see \cite[Proposition 2.4, Theorem 2.11]{mijatovic2012}.
	\item[3.] If \(s(+ \infty) = \infty\) and \(s(- \infty) < \infty\), then \(Q^*\)-a.s. \(H_{\rho} < \infty\) if and only if the second part in \eqref{cond -.3} holds, see \cite[Proposition 2.4, Theorem 2.11]{mijatovic2012}.
	\item[4.] If \(s(+ \infty) < \infty\) and \(s(- \infty) < \infty\), then \(Q^*\)-a.s. \(H_{\rho} < \infty\) if and only if the second parts in \eqref{cond -.3} and \eqref{cond +.3} hold, see \cite[Proposition 2.4, Theorem 2.11]{mijatovic2012}.
\end{enumerate}
We deduce the following version of \cite[Theorem 2.3]{MU(2012)} and \cite[Corollary 5.1]{Cherny2006} from Corollary \ref{coro: diff}.
\begin{corollary}
	\textup{(b.i), (b.ii)} and \textup{(b.iii)} from Corollary \ref{coro: diff} are equivalent to the following:
		\begin{enumerate}
		\item[\textup{(b.iv)}] One of the following conditions holds:
		\begin{enumerate}
			\item[\textup{(1)}] Lebesgue almost everywhere \(\beta = 0\).
			\item[\textup{(2)}]\eqref{cond +.3} and \eqref{cond -.1} hold.
			\item[\textup{(3)}] \eqref{cond +.1} and \eqref{cond -.3} hold.
			\item[\textup{(4)}] \eqref{cond +.3} and \eqref{cond -.3} hold.
		\end{enumerate}
	\end{enumerate} 
\end{corollary}
\begin{remark}
	Let us comment on the first case, which is closely related to the recurrence of \(P\).
	It is well-known that \(P\) is recurrent if and only if \(s(+ \infty) = \infty\) and \(s(- \infty) = \infty\), see \cite[Proposition 5.5.22]{KaraShre} or \cite[Theorem 5.1.1]{pinsky1995positive}. 
	Here, we call \(P\) recurrent if 
	\[
	P(X_t = y \text{ for some } t \geq 0) = 1 \text{ for all } y \in \mathbb{R}.
	\]
	In particular, \(P\) is conservative. 
	For recurrent diffusions, we have an ergodic theorem, i.e. \(P\)-a.s.
	\begin{align*}
	\frac{\int_0^t f(X_s)\dd s}{\int_0^t g(X_s)\dd s} \xrightarrow{\quad t \to \infty\quad} \frac{\int f(y) m(\dd y)}{\int g(y)m(\dd y)},  
	\end{align*}
	where \(m\) is the speed measure of \(X\) and \(f, g \colon \mathbb{R} \to [0, \infty)\) are Borel functions such that \(\int f(y)m(\dd y) < \infty\) and \(\int g(y)m(\dd y) > 0\). For a proof see \cite[Theorem 20.14]{Kallenberg}. If \(Q^* \ll P\), then the \(P\)-a.s. convergence transfers to \(Q^*\) and it follows that the speed measures of \(X\) coincides under \(P\) and \(Q^*\). Due to the specific choices of the drift and diffusion coefficients for \(P\) and \(Q^*\), this can only be if \(\beta = 0\) Lebesgue almost everywhere.  
\end{remark}
For explicit examples of all possible situations, we refer to \cite{Cherny2006}.
In summary, we have seen that Corollary \ref{coro:LE} is in line with the results proven in \cite{Cherny2006,MU(2012)}.

 \subsection{Absolute Continuity of It\^o-Jump-Diffusions}\label{sec: CFY}
In this section, we compare Corollary \ref{coro:LE} to the main result of Cheridito, Filipovi\'c and Yor \cite{CFY}. The proof in \cite{CFY} heavily relies on the concept of local uniqueness, which is in Markovian setups implied by the existence of unique solutions for all deterministic initial values, see \cite[Theorem III.2.40]{JS}.

Next, we recall a version of the setup of \cite{CFY}. We stress that \cite{CFY} includes a killing rate, which is not included in our case. Moreover, the underlying filtered spaces are different, such that the uniqueness assumptions are not identical, but very similar. 

Let \(b, \beta \colon \mathbb{R}^d \to \mathbb{R}^d\) and \(c \colon \mathbb{R}^d \to \mathbb{S}^d\) be Borel functions and let \(K\) be a Borel transition kernel from \(\mathbb{R}^d\) into \(\mathbb{R}^d\). Furthermore, let \(U \colon \mathbb{R}^d \times \mathbb{R}^d\to (0, \infty)\) be Borel. We extend these functions to \(\mathbb{R}^d_\Delta\) by setting them zero outside \(\mathbb{R}^d\). More precisely, we mean here the zero vector, the zero matrix etc.
In \cite{CFY} the following local boundedness assumptions are imposed:

	The maps \begin{center}\(b, b + c \beta, c, \int \left(1 \wedge \|y\|^2\right) K(\cdot, \dd y)\) and \(\int \left(1 + \|y\|^2\right)U(\cdot, y)K(\cdot, \dd y)\)\end{center} are locally bounded.

We set 
\begin{align*}
B &\triangleq \int_0^\cdot b (X_s) \dd s,\\
B' &\triangleq \int_0^\cdot \left(b (X_s) + c (X_s) \beta (X_s)\right) \dd s,\\
C &\triangleq \int_0^\cdot c(X_s) \dd s,\end{align*}
and \begin{align*}
\nu(\dd t \times \dd x) &\triangleq K(X_{t-}, \dd x)\dd t,\\
\nu'(\dd t \times \dd x) &\triangleq U(X_{t-}, x) K(X_{t-}, \dd x) \dd t.
\end{align*}
Let \(\rho_n\) and \(\rho\) be as in \eqref{eq: exit} and let \(\eta\) be a probability measure on \((\mathbb{R}^d, \mathscr{B}(\mathbb{R}^d))\).

In the following,
\(P\) is a solution to the SMP \((\rho; \eta; B, C, \nu)\) with \(\rho\)-localizing sequence \((\rho_n)_{n \in \mathbb{N}}\) and \(Q^*\) is a solution to the SMP \((\rho; \eta; B', C, \nu')\).

Define
\begin{align*}
H^* \triangleq \frac{1}{2} &\int_0^{\cdot \wedge \rho} \la \beta(X_t), c(X_t) \beta (X_t) \ra \dd t \\&+ \int_0^{\cdot \wedge \rho}\int\left(U(X_{t-}, x) \log(U(X_{t-}, x)) - U(X_{t-}, x) + 1\right) K(X_{t-}, \dd x)\dd t,
\end{align*}
see also Proposition \ref{prop: Nov}.
The main result in \cite{CFY} can be rephrased as follows:

If \(Q^*\) is the only solution to the SMP \((\rho; \eta; B', C, \nu')\) and 
\begin{align}\label{eq: loc novi}
E^P \left[ \exp\left(H^*_{\rho_n}\right)\right] < \infty \text{ for all } n \in \mathbb{N},
\end{align}
then a formula like \eqref{eq: CMG2} holds for all \((\mathscr{F}^o_t)_{t \geq 0}\)-stopping times \(\xi\), and \(Q^* \ll_{\textup{loc}} P\) holds if \(Q^*\) is conservative.

The condition \eqref{eq: loc novi} is a Novikov-type condition, which ensures that \(Z^{\rho_n}\) is a uniformly integrable \(P\)-martingale, where \(Z\) is defined as in Section \ref{sec:LE}, see also Standing Assumption \ref{SA1}. In particular, it implies that \(E^P \big[Z_{\rho_n}\big] = 1\), see Proposition \ref{prop: Nov} and Standing Assumption \ref{SA3a}.


Next, we compare this statement to Corollary \ref{coro:LE}.
Let \(H\) be defined as in Section \ref{sec:LE}, i.e. in this case
\begin{equation}\label{eq: H 2}
\begin{split}
H \triangleq \int_0^{\cdot \wedge \rho} &\la \beta (X_t), c(X_t) \beta(X_t) \ra \dd t \\&+ \int_0^{\cdot \wedge \rho} \int \left(1 - \sqrt{U(X_{t-}, x)}\right)^2 K(X_{t-},\dd x)\dd t.
\end{split}
\end{equation}
We note that \(H \leq 2 H^*\), which follows from the elementary inequality
\begin{align}\label{eq: elementary ineq}
\left(1 - \sqrt{x}\right)^2 \leq x \log(x) - x + 1\textup{ for all } x > 0.
\end{align}
Thus, \eqref{eq: loc novi} implies that \(P\)-a.s. \(H_{\rho_n} < \infty\), which yields that \(P\)-a.s. \(\rho \leq \sigma\). In this setting, it can be shown that \(P\)-a.s. \(\rho_n < \rho\), see \cite[Lemma 3.1 and the paragraph below its proof]{CFY}. Thus, \(P\)-a.s. \(\rho_n < \sigma\), i.e. Standing Assumption \ref{SA 2} (a) holds.
Consequently, Corollary \ref{coro:LE} implies that \(Q^*\ll_\textup{loc} P\) is true in the case where \(Q^*\) is conservative.
Furthermore, the formula \eqref{eq: CMG2} holds for all stopping times \(\xi\). In this regard, our result is different from the main result in \cite{CFY}, which only applies for stopping times of the canonical filtration \((\mathscr{F}^o_t)_{t \geq 0}\). 

\section{Absolute Continuity of Multidimensional Diffusions }\label{sec: Diff M}
While the one-dimensional diffusion case is almost fully understood, the literature on the multi-dimensional setting is less complete. 
In this section, we explain how to derive deterministic equivalent conditions for the (local) absolute continuity of multi-dimensional diffusions in a radial case and deterministic sufficient and necessary conditions for the absolute continuity for multi-dimensional diffusions with radial diffusion coefficient. The underlying idea is to compare the multidimensional diffusions with one-dimensional ones and then to use results on the finiteness of integral functionals as given in \cite{mijatovic2012}. This strategy is related to the idea behind Khasminskii's test for explosion, see \cite{ikeda1977} for details.

\subsection{The General Setting}
 We start by a formal introduction to the setting, which is very close to a multidimensional version of the setup studied in Section \ref{sec: Diff}.
As in Section \ref{sec: Diff}, \(\nu\) will always be the zero measure and we will remove it from all notations. 
 
 Let \(b\) and \(\beta\) be two Borel functions \(\mathbb{R}^d \to \mathbb{R}^d\) and \(c\) be a Borel function \(\mathbb{R}^d \to \mathbb{S}^d\), where \(\mathbb{S}^d\) denotes the set of all non-negative definite and symmetric real \(d \times d\) matrices. We extend these functions to \(\mathbb{R}^d_\Delta\) by setting them to the zero vector and the zero matrix, respectively.
 We set 
 \begin{align*}
 B &\triangleq \int_0^\cdot b (X_s) \dd s,\\ B' &\triangleq \int_0^\cdot \left(b(X_s) + (c\beta)(X_s)\right) \dd s,\\ C &\triangleq \int_0^\cdot c(X_s)\dd s.
 \end{align*}
Let \(\rho\) be as in \eqref{eq: exit}.
\begin{SAS}\label{SA 8}
	The functions \(b, b + c\beta\) and \(\la \beta, c\beta\ra\) are locally bounded and \(c\) is continuous such that \(\la y, c(x)y\ra > 0\) for all \(y \in \mathbb{R}^d\backslash \{0\}\) and \(x \in \mathbb{R}^d\).
\end{SAS}
For all \(x_0 \in \mathbb{R}^d\), this standing assumption implies that the SMP \((\rho; \delta_{x_0}; B, C)\) has a solution \(P\) and the SMP \((\rho; \delta_{x_0}; B', C)\) has a solution \(Q^*\). 
Furthermore, it follows as in Section \ref{sec: Diff} that the Standing Assumptions \ref{SA 2} and \ref{SA 3} are satisfied. In particular, all solutions to each of these SMPs coincide on \(\mathscr{F}\). 
We define the non-negative local \(P\)-martingale \(Z\) as in Section \ref{sec:LE}.
For all \(n \in \mathbb{N}\) the identity \(E^P\big[Z_{\rho_n}\big] = 1\) follows immediately from Novikov's condition and the assumption that \(\langle \beta, c\beta\rangle\) is locally bounded, see Proposition \ref{prop: Nov}. In other words, Standing Assumption \ref{SA3a} holds, too.
For proofs of the necessary facts we refer to \cite{pinsky1995positive}. 
Let us stress that the continuity assumption on \(c\) is important for high dimensional cases, see \cite{zbMATH01140123} for an example where \(c\) is uniformly elliptic but the corresponding SMP has more than one solution.

The following version of Corollary \ref{coro:LE} holds:
\begin{corollary}\label{coro: diff 2}
	\begin{enumerate}
		\item[\textup{(a)}]
				The following are equivalent:\begin{enumerate}
			\item[\textup{(a.i)}]
			\(Q^*\)-a.s. \(\int_0^{\rho} \la \beta(X_s), c(X_s)\beta(X_s)\ra\dd s <\infty\) on \(\{\rho < \infty\}\).
			\item[\textup{(a.ii)}]
			\(Z\) is a \(P\)-martingale.
			\item[\textup{(a.iii)}]
			\(Q^* \ll_{\textup{loc}} P\) with \(\frac{\dd Q^*}{\dd P}|_{\mathscr{F}_t} = Z_t.\)
		\end{enumerate}
	\item[\textup{(b)}]
		The following are equivalent:\begin{enumerate}
			\item[\textup{(b.i)}]
			\(Q^*\)-a.s. \(\int_0^{\rho} \la \beta(X_s), c(X_s)\beta(X_s)\ra\dd s <\infty\).
			\item[\textup{(b.ii)}]
			\(Z\) is a uniformly integrable \(P\)-martingale.
			\item[\textup{(b.iii)}]
			\(Q^* \ll P\) with \(\frac{\dd Q^*}{\dd P} = Z_\infty.\)
	\end{enumerate}
\end{enumerate}
\end{corollary}
Versions of the equivalences (a.i) \(\Leftrightarrow\) (a.ii) and (b.i) \(\Leftrightarrow\) (b.ii) have been derived \cite{RufSDE}.
In the next two sections, we use this result to deduce deterministic criteria.
\subsection{The Radial Case}
In this subsection we will consider the radial case. We will still assume that Standing Assumption \ref{SA 8} holds and that \(x_0 \not= 0\). In addition, we impose the following standing assumption.
\begin{SAS}
	There exist Borel functions \(\widehat{c} \colon [0, \infty) \to (0, \infty)\) and \(\widehat{b}\colon [0, \infty) \to \mathbb{R}\) such that for all \(x \in \mathbb{R}^d\)
	\begin{align*}
	\widehat{c} \left(\frac{\|x\|^2}{2}\right) &= \la x, c(x)x\ra,\\
	 \widehat{b} \left(\frac{\|x\|^2}{2}\right) &= \la x, (b + c \beta)(x)\ra + \frac{\textup{trace } c(x)}{2}.
	 \end{align*}
	 Furthermore, there exists a Borel function \(\hat{f} \colon (0, \infty) \to [0, \infty)\) such that for all \(x \in \mathbb{R}^d\backslash \{0\}\)
	 \begin{align*}
	 \widehat{f} \left(\frac{\|x\|^2}{2}\right) &= \la \beta(x), c(x)\beta(x)\ra.
	\end{align*}
\end{SAS}
We stress that Standing Assumption \ref{SA 8} implies that 
\begin{align}\label{eq: ES}
\frac{1 + |\widehat{b}| + \widehat{f}}{\widehat{c}} \in L^1_\textup{loc}((0, \infty)).\end{align}
Define 
\begin{align*}
W \triangleq \int_0^\cdot \frac{\la X_t, \dd X^c_t\ra}{\widehat{c}^{\frac{1}{2}} \left(\frac{\|X_t\|^2}{2}\right)} 
\end{align*}
on the random set \(\of 0, \rho\of\).
For \(t < \rho\), we deduce from Standing Assumption \ref{SA 8} that 
\begin{align*}
\lle W\rre_t = t.
\end{align*}
Thus, by \cite[Corollary 5.10]{J79}, we may extend \(W\) to continuous local \(P\)-martingale and by Knight's theorem, see \cite[Theorem 1.9]{RY}, we find a one-dimensional Brownian motion, possibly defined on an extension of our filtered probability space, which coincides with \(W\) on \(\of 0, \rho\of\). We denote this Brownian motion again by \(W\). 
An application of It\^o's formula yields that on \(\of 0, \rho\of\)
\begin{align*}
\dd \left(\frac{\|X_t\|^2}{2} \right)&= \la X_t, \dd X^c_t\ra + \widehat{b}\left(\frac{\|X_t\|^2}{2}\right)\dd t
\\&= \widehat{c}^\frac{1}{2}\left(\frac{\|X_t\|^2}{2}\right)\dd W_t + \widehat{b}\left(\frac{\|X_t\|^2}{2}\right)\dd t.
\end{align*}
Because the Engelbert-Schmidt conditions \eqref{eq: ES} are satisfied, there exists a \([0, \infty]\)-valued diffusion \(Y\) up to explosion with dynamics
\begin{align}\label{eq: SDE Y}
\dd Y_t = \widehat{c}^\frac{1}{2}(Y_t)\dd \widetilde{W}_t + \widehat{b}(Y_t) \dd t,\quad Y_0 = \tfrac{1}{2}\|x_0\|^2 \not = 0,
\end{align}
where \(\widetilde{W}\) is a one-dimensional Brownian motion, see \cite{KaraShre} for more details.
Here, explosion means exiting the interval \((0, \infty)\) and the explosion time of the diffusion \(Y\) is denoted by \(\theta\). Furthermore, the stochastic differential equation \eqref{eq: SDE Y} satisfies uniqueness in law. 
If the process \(Y\) does not explode to the origin, then \(\|X\|^2\) is always positive and the law of \(\frac{1}{2}\|X\|^2\) coincides with the law of \(Y\). In particular, we have 
\begin{align*}
\theta \overset{d}{=} \rho,\qquad \int_0^{\rho} \la \beta(X_t), c(X_t)\beta(X_t)\ra \dd t\ \overset{d}{=}\ \int_0^{\theta} \widehat{f}(Y_t)\dd t,
\end{align*}
where \(\overset{d}=\) indicates equality in law.
Now, we can deduce deterministic equivalent conditions for (a.i) and (b.i) of Corollary \ref{coro: diff 2} as in Section \ref{sec: Diff}.

For completeness, we state them formally: Set
\begin{equation}\label{eq: s multi}
\begin{split}
p (x) &\triangleq \exp \left( - \int_1^x \frac{2 \widehat{b}(y)}{\widehat{c}(y)} \dd y \right), \ x \in (0, \infty),
\\
s(x) &\triangleq \int_1^x p(y)\dd y,\ x \in (0, \infty),\\
s(+ \infty) &\triangleq \lim_{x \nearrow + \infty} s(x),
\\
s(0+) &\triangleq \lim_{x \searrow 0} s(x).
\end{split}
\end{equation}
We define the following conditions:
\begin{gather}
s(+ \infty) = \infty,\label{cond +.1, R}\\
s(+ \infty) < \infty\quad \textup{and}\quad \frac{s (+ \infty) - s}{p \widehat{c}} \not\in L^1_\textup{loc} (\infty),\label{cond +.2, R}\\
s(+ \infty) < \infty \quad\textup{and} \quad \frac{(s(+ \infty) - s)\widehat{f}}{p\widehat{c}} \in L^1_\textup{loc}(- \infty),\label{cond +.3, R}
\end{gather}
and similarly 
\begin{gather}
s(0+) = -\infty,\label{cond -.1, R}\\
s(0+) >- \infty\quad \textup{and}\quad \frac{s - s(0+)}{p\widehat{c}} \not\in L^1_\textup{loc} (-0),\label{cond -.2, R},\\
s(0+) >- \infty \quad\textup{and} \quad \frac{(s- s(0+))\widehat{f}}{p\widehat{c}} \in L^1_\textup{loc}(- 0).\label{cond -.3, R}
\end{gather}
Recall that, due to Feller's test for explosion, see \cite[Propositions 2.4, 2.5, 2.12]{mijatovic2012}, \(Y\) does not explode to the origin if and only if either \eqref{cond -.1, R} or \eqref{cond -.2, R} holds.
Now, Corollary \ref{coro:LE} and \cite[Proposition 2.4, Theorems 2.10 and 2.11]{mijatovic2012} imply the following result.
\begin{corollary}
	Suppose that either \eqref{cond -.1, R} or \eqref{cond -.2, R} holds.
	\begin{enumerate}
		\item[\textup{1.}] \textup{(a.i), (a.ii)} and \textup{(a.ii)} from Corollary \ref{coro: diff 2} are equivalent to the following: \begin{enumerate}
			\item[\textup{(a.iv)}] One of the conditions \eqref{cond +.1, R}, \eqref{cond +.2, R} and \eqref{cond +.3, R} holds.
		\end{enumerate}
	\item[\textup{2.}] 
		\textup{(b.i), (b.ii)} and \textup{(b.iii)} from Corollary \ref{coro: diff 2} are equivalent to the following:
	\begin{enumerate}
		\item[\textup{(b.iv)}] One of the following conditions holds:
		\begin{enumerate}
	\item[\textup{(1)}] Lebesgue almost everywhere \(\widehat{f} = 0\).
	\item[\textup{(2)}]\eqref{cond +.3, R} and \eqref{cond -.1, R} hold.
	\item[\textup{(3)}] \eqref{cond +.1, R} and \eqref{cond -.3, R} hold.
	\item[\textup{(4)}] \eqref{cond +.3, R} and \eqref{cond -.3, R} hold.
\end{enumerate}
	\end{enumerate} 
	\end{enumerate}
\end{corollary}
\subsection{The Partial Radial Case}
Next, we derive a Khasminskii-type test for the absolute continuity of two multi-dimensional diffusions. We still assume that Standing Assumption \ref{SA 8} holds. Furthermore, we impose the following standing assumption.

\begin{SAS} There exists a locally Lipschitz continuous function \(\tilde{c} \colon [0, \infty) \to (0, \infty)\) such that 
\begin{align}\label{assp radial}
\la x, c(x) x\ra &= \tilde{c} \left( \frac{\|x\|^2}{2}\right),\quad x \in \mathbb{R}^d,
\end{align}
and \(x_0 \not = 0\).
\end{SAS}
Let us formulate two conditions:
\begin{condition}\label{cond: md1}
	There exist a locally Lipschitz continuous functions \(v \colon (0, \infty) \to (0, \infty)\) such that for all \(x \in \mathbb{R}^d \backslash\{0\}\)
	\begin{align*}
	v\left(\frac{\|x\|^2}{2}\right) &\geq \frac{\textup{trace } c(x)}{2}  +  \la x, b(x) + c(x)\beta (x)\ra,
	\end{align*}
	and a decreasing Borel function \(w\colon (0, \infty) \to [0, \infty)\) such that for all \(x \in \mathbb{R}^d \backslash \{0\}\)
	\begin{align}\label{eq: w decre}
	\la \beta(x), c(x)\beta(x)\ra \geq w \left(\frac{\|x\|^2}{2}\right)
	\end{align}
	and \(\llambda(w > 0) > 0\), where \(\llambda\) denotes the Lebesgue measure.
\end{condition}
\begin{condition}\label{cond: md2}
	There exist a locally Lipschitz continuous functions \(v \colon (0, \infty) \to (0, \infty)\) such that for all \(x \in \mathbb{R}^d \backslash\{0\}\)
	\begin{align*}
	v\left(\frac{\|x\|^2}{2}\right) &\leq \frac{\textup{trace } c(x)}{2} +  \la x, b(x) + c(x)\beta (x)\ra,
	\end{align*}
	and an increasing Borel function \(w\colon (0, \infty) \to [0, \infty)\) such that for all \(x \in \mathbb{R}^d \backslash\{0\}\)
	\begin{align*}
	\la \beta(x), c(x)\beta(x)\ra \leq w \left(\frac{\|x\|^2}{2}\right).
	\end{align*}
\end{condition}
Let us discuss our strategy in the case where Condition \ref{cond: md1} holds.
We find a one-dimensional \([0, \infty]\)-valued diffusion \(Y\) whose paths are above those of \(\frac{1}{2}\|X\|^2\) till one of them explodes. Provided \(Y\) can only explode to \(+ \infty\), we have 
\(
\theta \leq \rho,
\)
where \(\theta\) is the explosion time of \(Y\).
Thus, using  \eqref{eq: w decre}, we obtain that 
\begin{equation}\label{eq: bound md}
\begin{split}
\int_0^{\rho} \la \beta(X_s), c(X_s)\beta (X_s)\ra \dd s &\geq \int_0^{\theta} \la \beta(X_s), c(X_s)\beta (X_s)\ra \dd s 
\\&\geq \int_0^\theta w \left(\frac{\|X_s\|^2}{2}\right) \dd s
\\&\geq \int_0^\theta w \left(Y_s\right) \dd s.
\end{split}
\end{equation}
In other words, \[\int_0^\theta w \left(Y_s\right) \dd s = \infty 
\] implies \[\int_0^{\rho} \la \beta(X_s), c(X_s)\beta (X_s)\ra \dd s = \infty 
.\]

By a similar argument, Condition \ref{cond: md2} can be used to obtain conditions for the finiteness of \(\int_0^{\rho} \la \beta(X_s), c(X_s) \beta(X_s)\ra \dd s\).

The remaining program of this section is to formulate these deterministic conditions, state the result and fill in the remaining details. 

Define 
\begin{align*}
p (x) &\triangleq \exp \left( - \int_1^x \frac{2 v(y)}{\tilde{c}(y)} \dd y \right), \ x \in (0, \infty),
\end{align*}
and let \(s, s(+ \infty)\) and \(s(0+)\) be as in \eqref{eq: s multi}.

Furthermore, we define the following conditions:
\begin{gather}
s(+ \infty) = \infty,\label{md cond +.1}\\
s(+ \infty) < \infty\quad \textup{and}\quad \frac{s (+ \infty) - s}{p \tilde{c}} \not\in L^1_\textup{loc} (\infty),\label{md cond +.2}\\
s(+ \infty) < \infty \quad\textup{and} \quad \frac{(s(+ \infty) - s)w}{p \tilde{c}} \not \in L^1_\textup{loc}(- \infty),\label{md cond +.3}\\
s(+ \infty) < \infty \quad\textup{and} \quad \frac{(s(+ \infty) - s)w}{p \tilde{c}} \in L^1_\textup{loc}(- \infty),\label{md cond +.4}
\end{gather}
and similarly 
\begin{gather}
s(0+) = -\infty,\label{md cond -.1}\\
s(0+) >- \infty\quad \textup{and}\quad \frac{s - s(0+)}{p \tilde{c}} \not\in L^1_\textup{loc} (-0),\label{md cond -.2}\\
s(0+) >- \infty \quad\textup{and} \quad \frac{(s- s(0+))w}{p \tilde{c}} \not \in L^1_\textup{loc}(- 0)\label{md cond -.3},\\
s(0+) >- \infty \quad\textup{and} \quad \frac{(s- s(0+))w}{p \tilde{c}} \in L^1_\textup{loc}(- 0)\label{md cond -.4}.
\end{gather}
We obtain a deterministic test for two multi-dimensional diffusions to be absolutely continuous.
\begin{proposition}
	\begin{enumerate}
		\item[\textup{(i)}]
		Suppose that Condition \ref{cond: md1} holds and that one of the following conditions holds:
		\begin{enumerate}
			\item[\textup{(i.a)}]
			\eqref{md cond +.1} and \eqref{md cond -.1} hold.
			\item[\textup{(i.b)}]
			\eqref{md cond +.3} and \eqref{md cond -.1} hold.
			\item[\textup{(i.c)}]
			\eqref{md cond +.1}, \eqref{md cond -.2} and \eqref{md cond -.3} hold.
			\item[\textup{(i.d)}]
			\eqref{md cond +.3}, \eqref{md cond -.2} and \eqref{md cond -.3} hold.
		\end{enumerate}
	Then, \(Q^* \not\hspace{-0.055cm} \ll P\).
		\item[\textup{(ii)}]
				Suppose that Condition \ref{cond: md2} holds and that one of the following conditions holds:
		\begin{enumerate}
			\item[\textup{(ii.a)}]
			\eqref{md cond +.4} and \eqref{md cond -.1} hold.
			\item[\textup{(ii.b)}]
			\eqref{md cond +.1}, \eqref{md cond -.2} and \eqref{md cond -.4} hold.
			\item[\textup{(ii.c)}]
			\eqref{md cond +.4}, \eqref{md cond -.2} and \eqref{md cond -.4} hold.
		\end{enumerate}
	Then, \(Q^* \ll P\) with \(\frac{\dd Q}{\dd P} = Z_\infty\).
	\end{enumerate}
\end{proposition}
\begin{proof}
By It\^o's formula, we obtain on \(\of 0, \rho\of\) 
\begin{align*}
\dd \left(\frac{\|X_t\|^2}{2}\right)&= \la X_t, \dd X^c_t\ra +  \left( \la X_t, (b + c\beta)(X_t)\ra + \frac{\textup{trace } c(X_t)}{2}\right)\dd t, \\ \frac{\|X_0\|^2}{2} &= \frac{\|x_0\|^2}{2}.
\end{align*}
Define 
\begin{align*}
W \triangleq \int_0^\cdot \frac{\la X_t, \dd X^c_s\ra}{\tilde{c}^{\frac{1}{2}} \left(\tfrac{\|X_t\|^2}{2}\right)}
\end{align*}
on the random set \(\of 0, \rho\of\).
For \(t < \rho\), we deduce from our radial assumption \eqref{assp radial} that 
\begin{align*}
\lle W\rre_t = t.
\end{align*}
Thus, by \cite[Corollary 5.10]{J79}, we may extend \(W\) to a continuous local \(P\)-martingale and by Knight's theorem, see \cite[Theorem 1.9]{RY}, we find a one-dimensional Brownian motion, possibly defined on an extension of our filtered probability space, which coincides with \(W\) on \(\of 0, \rho\of\). We denote this Brownian motion again by \(W\).
Because we might work on an extension, we will drop \(P\) from our notation. The null sets in the following correspond to the extension.
We have on \(\of 0, \rho\of\)
\begin{align*}
\dd \left(\frac{\|X_t\|^2}{2}\right) &= \tilde{c}^{\frac{1}{2}}\left(\frac{\|X_t\|^2}{2}\right) \dd W_t +  \left( \la X_t, (b + c \beta)(X_t)\ra + \frac{\textup{trace } c(X_t)}{2}\right)\dd t, \\ \frac{\|X_0\|^2}{2} &= \frac{\|x_0\|^2}{2}.
\end{align*}
Since stochastic differential equations with locally Lipschitz continuous coefficients satisfy pathwise uniqueness and pathwise uniqueness together with weak-existence implies strong existence, there exists a \([0, \infty]\)-valued process \(Y\) with dynamics
\begin{align*}
\dd Y_t = \tilde{c}^{\frac{1}{2}} (Y_t)\dd W_t + v(Y_t) \dd t,\quad Y_0 = \tfrac{1}{2}\|x_0\|^2 \not = 0, 
\end{align*}
up to explosion, see \cite[Theorem 5.2.5, Corollaries 5.3.23, 5.5.16]{KaraShre} for details. Here, explosion has to be understood as exiting the interval \((0, \infty)\) and the explosion time of the diffusion \(Y\) is denoted by \(\theta\).
We stress that \(W\) is the same Brownian motion as defined above.
In the following, we turn to the individual cases (i) and (ii).
\begin{enumerate}
\item[\textup{(i)}] It follows from the classical comparison result of Ikeda and Watanabe, see \cite[Theorem IX.3.7]{RY}, that a.s.
\begin{align}\label{comp: 1}
\tfrac{1}{2}\|X\|^2 \leq Y \textup{ on } \of 0, \rho \wedge \theta\of.
\end{align}
Since in all cases (i.a) -- (i.d) either \eqref{md cond -.1} or \eqref{md cond -.2} holds, Feller's test for explosion yields that \(Y\) can only explode to \(+ \infty\), i.e. up to a null set 
\begin{align*}
\theta = \inf(t \geq 0 \colon Y_t = \infty).
\end{align*}
Thus, \eqref{comp: 1} yields that a.s. \(\theta \leq \rho\). Now, recalling \eqref{eq: bound md} and Corollary \ref{coro: diff 2}, it suffices to verify that (i.a) -- (i.d) imply that a.s.
\begin{align*}
\int_0^{\theta} w(Y_s)\dd s = \infty.
\end{align*}
This follows case by case from \cite[Propositions 2.3, 2.4, Theorems 2.10, 2.11]{mijatovic2012}.
\item[\textup{(ii)}]
Using once again the comparison result of Ikeda and Watanabe, we obtain a.s.
\begin{align*}
Y \leq \tfrac{1}{2}\|X\|^2\textup{ on } \of 0, \rho \wedge \theta\of.
\end{align*}
Since in all cases (ii.a) -- (ii.c) either \eqref{md cond -.1} or \eqref{md cond -.2} holds, it follows as in (ii) that a.s. \(\rho \leq \theta\). We obtain that
\begin{equation*}
\begin{split}
\int_0^{\rho} \la \beta(X_s), c(X_s)\beta (X_s)\ra \dd s
&\leq \int_0^\rho w \left(\frac{\|X_s\|^2}{2}\right) \dd s
\\&\leq \int_0^\rho w \left(Y_s\right) \dd s
\\&\leq \int_0^\theta w \left(Y_s\right) \dd s,
\end{split}
\end{equation*}
and the claim follows again from Corollary \ref{coro: diff 2} and \cite[Propositions 2.3, 2.4, Theorems 2.10, 2.11]{mijatovic2012}.
\end{enumerate}
The proof is complete.
\end{proof}
In the following section we present an application of Theorem \ref{theo:main1} without any uniqueness assumption.


\section{Martingale Property of Stochastic Exponentials}\label{sec: BC}
In the previous sections we have seen applications of Theorem \ref{theo:main1} under a uniqueness assumption. In this section, we show that also without such an assumption Theorem \ref{theo:main1} has interesting consequences. 

We illustrate this by deriving a generalization of the classical linear growth condition of Bene\u s \cite{doi:10.1137/0309034} to general continuous It\^o-processes. Let us shortly explain the idea.
If a local martingale has a localizing sequence, which is also a localizing sequence for a modified SMP, then the local martingale is a true martingale. In the following, we will formulate conditions which imply the existence of such a localizing sequence for any solution of the modified SMP. Thus, no uniqueness assumption is required.

Let us shortly recall the result of Bene\u s \cite{doi:10.1137/0309034}: Assume that \(W\) is a  \(d\)-dimensional Brownian motion and \(\mu\) is an \(\mathbb{R}^d\)-valued predictable process on the Wiener space. Then, the stochastic exponential
\[
\exp\left( \int_0^\cdot \langle \mu_s(W), \dd W_s\rangle - \frac{1}{2} \int_0^\cdot \|\mu_s(W)\|^2 \dd s\right)
\]
is a martingale if \(\mu\) is at most of linear growth. We refer to \cite[Corollary 3.5.16]{KaraShre} for a precise statement.

In the following we generalize this result to cases where \(W\) may be a continuous It\^o-process. Of course, it is possible to allow additionally jumps. However, we think that focusing on the less technical continuous setup suffices to explain the main idea.
For similar conditions in a Markovian jump-diffusion setup we refer to \cite{doi:10.1137/S0040585X97986382}.

Since \(\nu\) will always be the zero measure we remove it from all notations.
Let \(b\) and \(\beta\) be \(\mathbb{R}^d\)-valued predictable processes and \(c\) be a predictable process with values in \(\mathbb{R}^d \otimes \mathbb{R}^d\). 
We set 
\begin{align*}
B &\triangleq \int_0^\cdot b_s \dd s,\\
B' &\triangleq \int_0^\cdot \left(b_s + c_s \beta_s\right) \dd s,\\
C &\triangleq \int_0^\cdot c_s \dd s.\end{align*}
Let \(\rho_n\) and \(\rho\) be as in \eqref{eq: exit} and let \(\eta\) be a probability measure on \((\mathbb{R}^d, \mathscr{B}(\mathbb{R}^d))\).
\begin{SA}\label{SA 9}
	Let \(P\) be a solution to the SMP \((\rho; \eta; B, C)\) with \((\rho_n)_{n \in \mathbb{N}}\) as in \eqref{eq: exit} as \(\rho\)-localizing sequence. Let \(\sigma\) be as in \eqref{eq: sigma}. Furthermore, \(P\)-a.s. \(\rho_n < \sigma\) for all \(n \in \mathbb{N}\) and \(\int \|x\|^2 \eta(\dd x) < \infty\). Define \(Z\) as in Section \ref{sec:LE}. It holds that \(E^P\big[Z_{\rho_n}\big] = 1\) for all \(n \in \mathbb{N}\).
\end{SA}

\begin{corollary}\label{coro:mart}
Suppose 
there exists a Borel function \(\gamma \colon [0, \infty) \to [0, \infty)\) such that \(\int_0^T \gamma(s)\dd s < \infty\) for all \(T \geq 0\) and for all continuous functions \(\omega \colon [0, \infty) \to \mathbb{R}^d\) it holds that
\begin{align*}
\| b_t(\omega) + \beta_t(\omega) c_t(\omega)\|^2 &\leq \gamma(t) \left(1 + \sup_{s \in [0, t]} \|\omega(s)\|^2\right),\\
 \textup{ trace } c_t(\omega)&\leq \gamma(t) \left(1 + \sup_{s \in [0, t]} \|\omega(s)\|^2\right).
\end{align*}
Then \(Z\) is a \(P\)-martingale.
\end{corollary}
\begin{proof}
By Theorem \ref{theo:main1}, it suffices to show that for all solutions \(Q\) to the SMP \((\rho; \eta; B',C)\) we have \(\q(\rho = \infty) = 1\). 

It is not difficult to see that, due to our linear growth conditions, we find a constant \(k(t)\), which only depends on \(t\), such that 
\begin{align*}
\E^\q\bigg[ \sup_{s \in [0, t \wedge \rho_n]} \|X_{s}\|^2\bigg] 
&\leq k(t) \left(1  + \E^\q\left[ \int_0^t \gamma(s) \sup_{r \in [0, s \wedge \rho_n]} \|X_{r}\|^2\dd s\right]\right).
\end{align*}
Now, we deduce from Gronwall's lemma, see \cite[Lemma A.2.35]{Bichteler02}, that
\[
\E^\q\bigg[ \sup_{s \in [0, t \wedge \rho_n]} \|X_{s}\|^2\bigg]  \leq \textup{const. independent of } n.
\]
Using Chebyshev's inequality, we deduce that 
\[
\q(\rho_n \leq t) = Q\left(\sup_{s \in [0, t \wedge \rho_n]} \|X_{s}\| \geq n\right) \leq \frac{\textup{const. independent of } n}{n^2} \to 0 
\]
as \(n \to \infty\).
Since this holds for all \(t \geq 0\), we conclude that \(\q(\rho = \infty) = 1\) and the proof is complete. 
\end{proof}

Under the Engelbert-Schmidt conditions we have already seen equivalent conditions for the martingale property of \(Z\), see Section \ref{sec: Diff} or \cite{MU(2012)}. The linear growth condition presented in Corollary \ref{coro:mart} is not necessary. However, it applies in multi-dimensional setups, in non-Markovian cases and does not require any uniqueness assumption. 
Furthermore, it is typically easy to verify. 

\section*{Acknowledgments}
	 The authors thank Jean Jacod for fruitful discussions. Furthermore, the authors are very grateful to the anonymous referee, who's comments helped to improved the manuscript substantially. 
\bibliographystyle{plain}
\bibliography{References}

\end{document}